\documentclass[11pt]{amsart}
\usepackage{amsmath,amsthm,amsfonts,amssymb,times}
\usepackage{latexsym}

\usepackage{versions}
\usepackage{mathrsfs} 

\usepackage{turnstile}

\numberwithin{equation}{section}  

\DeclareMathAlphabet{\curly}{U}{rsfs}{m}{n}  

\theoremstyle{remark}
\newtheorem{remark}{Remark}
\theoremstyle{plain}
\newtheorem{prop}{Proposition}
\newtheorem{lem}{Lemma}[section]
\newtheorem{thm}{Theorem}

\newtheorem{definition}{Definition}

\newtheorem*{prop64new}{Proposition 6.4'}
\newtheorem*{lem13new}{Lemma 1.3'}
\newtheorem*{gt-maintheorem}{Theorem A \cite[Main Theorem]{gt-linearprimes}}

\numberwithin{equation}{section}

\newcommand{\Z}{\mathbb{Z}}
\newcommand{\R}{\mathbb{R}}
\newcommand{\E}{\mathbb{E}}  
\newcommand{\PR}{\mathbb{P}}  

\makeatletter
\renewcommand{\pmod}[1]{\allowbreak\mkern7mu({\operator@font mod}\,\,#1)}
\makeatother

\newcommand{\bal}{\[\begin{aligned}}
\newcommand{\eal}{\end{aligned}\]}

\newcommand{\be}{\begin{equation}}
\newcommand{\ee}{\end{equation}}

\newcommand{\del}{\ensuremath{\delta}}
\newcommand{\eps}{\ensuremath{\varepsilon}}
\newcommand{\g}{\gamma}

\renewcommand{\le}{\leqslant}
\renewcommand{\leq}{\leqslant}
\renewcommand{\ge}{\geqslant}
\renewcommand{\geq}{\geqslant}

\renewcommand{\(}{\left(}
\renewcommand{\)}{\right)}

\newcommand{\pfrac}[2]{\left(\frac{#1}{#2}\right)}  

\newcommand{\asym}{\sim}   
\newcommand{\rel}{\nsststile{}{}}  
\newcommand{\relr}{\nsdtstile{}{}} 
\newcommand{\relra}{\nsdtstile{}{\vect{\mathbf{a}}}} 
\newcommand{\vect}[1]{{\ensuremath{\vec{#1}}}}

\newcommand{\PP}{\mathcal{P}}
\newcommand{\QQ}{\mathcal{Q}}
\newcommand{\cS}{\mathcal{S}}
\newcommand{\cR}{\mathcal{R}}

\newcommand{\vol}{\operatorname{vol}}
\newcommand{\main}{\operatorname{main}}
\newcommand{\err}{\operatorname{err}}

\begin{document}

\title{Large gaps between consecutive prime numbers}

\author{Kevin Ford}
\address{Department of Mathematics\\ 1409 West Green Street \\ University
of Illinois at Urbana-Champaign\\ Urbana, IL 61801\\ USA}
\email{ford@math.uiuc.edu}

\author{Ben Green}
\address{Mathematical Institute\\
Radcliffe Observatory Quarter\\
Woodstock Road\\
Oxford OX2 6GG\\
England }
\email{ben.green@maths.ox.ac.uk}

\author{Sergei Konyagin}
\address{Steklov Mathematical Institute\\
8 Gubkin Street\\
Moscow, 119991\\
Russia}
\email{konyagin@mi.ras.ru}

\author{Terence Tao}
\address{Department of Mathematics, UCLA\\
405 Hilgard Ave\\
Los Angeles CA 90095\\
USA}
\email{tao@math.ucla.edu}

\begin{abstract}Let $G(X)$ denote the size of the largest gap between consecutive primes below $X$. Answering a question of Erd\H{o}s, we show that  \[ G(X) \geq f(X) \frac{\log X \log \log X \log \log \log \log X}{(\log \log \log X)^2},\] where $f(X)$ is a function tending to infinity with $X$.  Our proof combines existing arguments with a random construction covering a set of primes by arithmetic progressions.  As such, we rely on recent work on the existence and distribution of long arithmetic progressions consisting entirely of primes.
\end{abstract}

\maketitle

\setcounter{tocdepth}{1}
\tableofcontents


\section{Introduction}

Write $G(X)$ for the maximum gap between consecutive primes less than $X$. It is clear from the prime number theorem that
\[ G(X) \geq (1 + o(1)) \log X,\]
as the \emph{average} gap between the prime numbers which are $\le X$ is $\asym \log X$. 
In 1931, Westzynthius \cite{West} proved that infinitely often, the gap between consecutive 
prime numbers can be an arbitrarily large multiple of the average gap, that is,
$G(X)/\log X\to \infty$ as $X\to\infty$.  Moreover, he proved the qualitative
bound\footnote{As usual in the subject, $\log_2 x = \log \log x$, $\log_3 x = \log \log \log x$, and so on.  The conventions for asymptotic notation such as $\ll$ and $o()$ will be defined in Section \ref{not-sec}.}
\[
G(X) \gg \frac{\log X \log_3 X}{\log_4 X}.
\]
In 1935 Erd\H{o}s \cite{erdos-gaps} improved this to
\[ G(X) \gg \frac{\log X \log_2X}{(\log_3X)^2}\] and in 1938 Rankin \cite{R1} 
made a subsequent improvement
\[ G(X) \geq (c + o(1)) \frac{\log X \log_2 X \log_4 X}{(\log_3 X)^2}\] with $c = \frac{1}{3}$. The constant $c$ was subsequently improved several times: to $\frac{1}{2}e^{\gamma}$ by Sch\"onhage \cite{schonhage}, then to $c = e^{\gamma}$ by Rankin \cite{rankin-1963}, $c = 1.31256 e^{\gamma}$ by Maier and Pomerance \cite{MP} and, most recently, $c = 2e^{\gamma}$ by Pintz \cite{P}.

Our aim in this paper is to show that $c$ can be taken arbitrarily large.

\begin{thm}\label{mainthm} Let $R>0$.  Then for any sufficiently large $X$,
there are at least
$$R\frac{\log X \log_2 X \log_4 X}{(\log_3 X)^2}$$
consecutive composite natural numbers not exceeding $X$.
\end{thm}

In other words, we have
$$ G(X) \geq f(X) \frac{\log X \log_2 X \log_4 X}{(\log_3 X)^2}$$
for some function $f(X)$ that goes to infinity as $X \to \infty$.
Theorem \ref{mainthm} settles in the affirmative a long-standing conjecture of
Erd\H os \cite{Erd90}.

Theorem \ref{mainthm} has been simultaneously and independently established by Maynard \cite{maynard} by a different method (relying on the sieve-theoretic techniques related to those used recently in \cite{maynard-small} to obtain bounded gaps between primes, rather than results on linear equations between primes).  As it turns out, the techniques of this paper and of that in \cite{maynard} may be combined to establish further results on large prime gaps; see the followup paper \cite{FGKMT} to this work and to \cite{maynard}.

Based on a probabilistic model of primes, Cram\'er \cite{Cra} conjectured
that\footnote{Cram\'er is not entirely explicit with this conjecture.
In \cite{Cra}, he shows that his random analogues $P_n$ of
primes satisfy $\limsup (P_{n+1}-P_n)(\log P_n)^{-2}=1$ and writes ``Obviously we
may take this as a suggestion that, for the particular sequence of
ordinary prime numbers $p_n$, some similar relation may hold''.}
\[
\limsup_{X\to\infty} \frac{G(X)}{\log^2 X} = 1, 
\]
and Granville \cite{Gra}, using a refinement of Cram\'er's model,
has conjectured that the $\limsup$ above is in fact at least $2e^{-\gamma}=1.1229\ldots$.  These conjectures are well beyond the reach of our methods.
Cram\'er's model also predicts that the normalized prime gaps
$\frac{p_{n+1}-p_n}{\log p_n}$ should have exponential distribution, that is,
$p_{n+1}-p_n \ge C\log p_n$ for about $e^{-C}\pi(X)$ primes $\le X$.
Numerical evidence from prime calculations up to 
$4\cdot 10^{18}$ \cite{numerical-tos}
matches this prediction quite closely, with the exception of values of $C$
close to $\log X$, in which there is very little data available.  In fact,
$\max_{X\le 4\cdot 10^{18}} G(X)/\log^2 X \approx 0.9206$, slightly below the
predictions of Cram\'er and Granville.

Unconditional upper bounds for $G(X)$ are far from the conjectured
truth, the best being
$G(X) \ll X^{0.525}$ and due to Baker, Harman and Pintz \cite{BHP}.
Even the Riemann Hypothesis only\footnote{Some slight improvements are available if one also assumes some form of the pair correlation conjecture; see \cite{heath}.} furnishes the bound $G(X)\ll X^{1/2}\log X$ \cite{Cra1920}.

All works on lower bounds for $G(X)$ have followed a similar overall plan
of attack: show that there are at least $G(X)$ consecutive
integers in $(X/2,X]$, 
each of which has a ``very small'' prime factor.  
To describe the results, we make the following definition.

\begin{definition}\label{y-def}
Let $x$ be a positive integer. Define $Y(x)$ to be the largest integer $y$ for which one may select residue classes $a_p \pmod{p}$, one for each prime $p \leq x$, which together ``sieve out'' (cover) the whole interval $[y] = \{1,\dots,y\}$.
\end{definition}

The relation between this function $Y$ and gaps between primes is encoded in the following simple lemma.

\begin{lem}\label{lem11}
Write $P(x)$ for the product of the primes less than or equal to $x$. Then we have $G(P(x)+ Y(x) + x) \geq Y(x)$ for all $x$.
\end{lem}
\begin{proof}
Set $y = Y(x)$, and select residue classes $a_p \pmod{p}$, one for each prime $p \leq x$, which cover $[y]$. By the Chinese remainder theorem there is some $m$, $x <  m \leq x +P(x)$, with $m \equiv -a_p \pmod{p}$ for all primes $p \leq x$. We claim that all of the numbers $m+1,\dots, m+y$ are composite, which means that there is a gap of length $y$ amongst the primes less than $m+y$, thereby concluding the proof of the lemma. To prove the claim, suppose that $1 \leq t \leq y$. Then there is some $p$ such that $t \equiv a_p \pmod{p}$, and hence $m + t \equiv -a_p + a_p \equiv 0 \pmod{p}$, and thus $p$ divides $m + t$. Since $m+t > m > x \geq p$, $m + t$ is indeed composite.
\end{proof}

By the prime number theorem we have $P(x) = e^{(1 + o(1))x}$. It turns out (see below) that $Y(x)$ has size $x^{O(1)}$. Thus the bound of Lemma \ref{lem11} implies that
\[ G(X) \geq Y\big((1 + o(1)) \log X\big)\]
as $X \to \infty$.  Theorem \ref{mainthm} follows from this and the following bound for $Y$, the proof of which is the main business of the paper.
\begin{thm}\label{mainthm-y}
For any $R > 0$ and for sufficiently large $x$ we have
\begin{equation}\label{yxdef}
 Y(x) \geq R \frac{x \log x \log_3 x}{(\log_2 x)^2}.
\end{equation}
\end{thm}

The function $Y$ is intimately related to \emph{Jacobsthal's function} $j$. If $n$ is a positive integer then $j(n)$ is defined to be the maximal gap between integers coprime to $n$. In particular $j(P(x))$ is the maximal gap between numbers free of prime factors $\leq x$, or equivalently $1$ plus the longest string of consecutive integers, each divisible by some prime $p \leq x$.  The construction given in the proof of Lemma \ref{lem11} in fact proves that 
\[ j(P(x)) \geq Y\big((1 + o(1)) \log P(x)\big) = Y\big((1 + o(1)) x\big).\]
This observation, together with results in the literature, gives upper bounds for $Y$. The best upper bound known is $Y(x) \ll x^2$, which comes from Iwaniec's work \cite{Iw} on Jacobsthal's function. It is conjectured by Maier and Pomerance that in fact $Y(x) \ll x (\log x)^{2 + o(1)}$. This places a serious (albeit conjectural) upper bound on how large gaps between primes we can hope to find via lower bounds for $Y(x)$: a bound in the region of $G(X) \gtrapprox \log X (\log \log X)^{2 + o(1)}$, far from Cram\'er's conjecture, appears to be the absolute limit of such an approach.

We turn now to a discussion of the proof of Theorem \ref{mainthm-y}. Recall that our task is to find $y$, as large as possible, so that the whole interval $[y]$ may be sieved using congruences $a_p \pmod{p}$, one for each prime $p \leq x$. Prior authors divided the sieving into different steps, a key to all of
them being to take a common value of $a_p$ for ``large'' $p$,
say $a_p=0$ for $z<p<\del x$, where $\delta>0$ is a small constant and $z=x^{c\log_3 x/\log_2 x}$ for
some constant $c>0$.
The numbers
in $[y]$ surviving this first sieving either have all of their prime factors 
$\le z$ (i.e., they are ``$z$-smooth'') or are of the form $pm$ with $p$ prime
and $m\le y/\del x$.  One then appeals
to bounds for smooth numbers, e.g. \cite{deB},
to see that there are very few numbers of the first kind, say $O(x/\log^2 x)$.
By the prime number theorem there are
 $\asym y\log_2 x/\log x$ unsieved numbers of the second kind.
By contrast, if one were to take a random choice for $a_p$ for 
$z<p<\del x$, then with high probability, the number of unsifted integers
in $[y]$ would be considerably larger, about
$y\log z/\log x$.

One then performs a second sieving, choosing $a_p$ for ``small'' $p\le z$. 
Using a greedy algorithm, for instance, one can easily sieve out all but
\[
\frac{y\log_2 x}{\log x} \prod_{p\le z} \(1-\frac{1}{p}\) \asym
e^{-\gamma} \frac{y\log_2 x}{\log x \log z}
\]
of the remaining numbers.  There are alternative approaches using explicit
choices for $a_p$; we will choose our $a_p$ at random.  (The set $V$ of numbers surviving this second sieving has about the same size in each case.)

 If $|V| \le \pi(x)-\pi(\del x)$, the number of 
``very large'' primes, then we perform a (rather trivial) third sieving as follows: each $v \in V$ can be matched with
one of these primes $p$, and one may simply take $a_p=v$. 
This is the route followed by all authors up to and including Rankin \cite{rankin-1963}; improvements 
to $G(x)$ up to this point depended on improved bounds for counts of 
smooth numbers.  
The new idea introduced by Maier and Pomerance \cite{MP} was to make
the third sieving more efficient (and less trivial!) by using many $p
\in (\del x, x]$ to sift not one but \emph{two} elements of $V$. 
To do this they established a kind of 
``twin primes on average'' result implying that for most
$p\in (\del x,x]$, there are many pairs of elements of $V$ that are
 congruent modulo $p$.  Then the authors proved a crucial combinatorial result
 that \emph{disjoint} sets $V_p$ exist, each of two elements congruent
 modulo $p$, for a large proportion of these primes $p$; that is, for
 a large proportion of $p$, $a_p\mod p$ will sift out two elements of
 $V$, and the sifted elements are disjoint.
Pintz \cite{P} proved a ``best possible'' version of the combinatorial
result, that in fact one can achieve a ``nearly perfect matching'',
that is, disjoint sets $V_p$ for almost all primes $p\in (\del x,x]$,
and this led to the heretofore best lower bound for $G(X)$.

Heuristically, much more along these lines should be possible. With $y$ comparable to the right-hand side of \eqref{yxdef}, the set $V$ turns out have expected cardinality comparable to a large multiple of $x/\log x$.  Assuming that $V$ is a ``random'' subset of $[y]$, for every 
prime $p\in (\del x,x]$ there should in fact be a residue class $a\pmod p$ containing
$\gg \log x/(\log_2 x)^{O(1)}$ elements of $V$. (Roughly, the heuristic
predicts that the sizes of the sets $V \cap (a\pmod p)$ 
are Poisson distributed with parameter $\approx |V|/p$.) Whilst we
cannot establish anything close to this, we are able to use almost all
primes $p \in (x/2,x]$ to sieve $r$ elements of $V$, for any fixed
  $r$. Where Maier and Pomerance appealed to (in fact proved) a result
  about pairs of primes on average, we use results about arithmetic
  progressions of primes of length $r$, established in work of the
  second and fourth authors \cite{gt-linearprimes},  \cite{gt-nilmobius} and of these authors and
  Ziegler \cite{GTZ}. Specifically, we need results about progressions $q, q+ r!
  p, q+ 2r! p, \dots, q + (r-1)r! p$; if one ignores the technical factor $r!$, these are
  ``progressions of primes with prime common difference''. By taking
  $a_p = q$, the congruence $a_p \pmod{p}$ allows us to sift out all
  $r$ elements of such a progression, and it is here that we proceed
  more efficiently than prior authors. 
Ensuring that many of these $r$-element sifted sets are disjoint (or
at least have small intersections) is a rather difficult
problem, however.  Rather than dealing with these
intersections directly, we utilize the
random choice of $a_p$ in the second step to prove that with high 
probability, $V$ has a certain regularity with respect to 
intersections with progressions of
the form  $q, q+ r! p, q+ 2r! p, \dots, q + (r-1)r! p$.  We then prove
that most elements of $V$ survive the third sieving with uniformly
small probability.

\subsection{Acknowledgments}

The research of SK was partially performed while he was
visiting KF at the University of Illinois at 
Urbana--Champaign.  Research of KF and SK was
also carried out in part at the University of Chicago.
KF and SK are thankful to Prof. Wilhelm Schlag
for hosting these visits.  KF also thanks the
hospitality of the Institute of Mathematics and Informatics of the
Bulgarian Academy of Sciences.

Also, the research of the SK and TT was partially
performed while SK was 
visiting the Institute for Pure and Applied Mathematics (IPAM) at UCLA,
which is supported by the National Science Foundation.

KF was supported by NSF grant DMS-1201442.
BG was supported by ERC Starting Grant 279438, \emph{Approximate algebraic structure}. 
TT was supported by a Simons Investigator grant, the
James and Carol Collins Chair, the Mathematical Analysis \&
Application Research Fund Endowment, and by NSF grant DMS-1266164. 

Finally, all four of us wish to thank Andrew Granville and James Maynard for helpful discussions.

\subsection{Notational conventions}\label{not-sec}
  
We use $f=O(g)$ and $f\ll g$ to denote the claim that there is a constant $C>0$ such that
$|f(\cdot)| \le C g(\cdot)$ for all $\cdot$ in the domain of $f$.
We adopt the convention that $C$ is independent of any parameter
unless such dependence is indicated by subscript such as  $\ll_u$,
except that $C$ may depend on the parameter $r$ (which we consider to
be fixed) in Sections
\ref{sec:prime-progressions}--\ref{sec:prob} and 
\ref{first-lemma-sec}--\ref{second-lemma-sec}.

In  Sections
\ref{sec:prime-progressions}--\ref{sec:prob} and
\ref{first-lemma-sec}--\ref{second-lemma-sec}, 
the symbol $o(1)$ will stand for a function which tends
to $0$ as $x\to\infty$, uniform in all parameters except $r$ unless 
otherwise indicated.
The same convention applies to the asymptotic
notation
$f(x) \asym g(x)$, which means $f(x)=(1+o(1))g(x)$.
In Sections \ref{dickson-shifted-sec} and the Appendix, $o(g(N))$
refers to some function $h(N)$ satisfying $\lim_{N\to\infty} h(N)/g(N)=0$.

The symbols $p$, $q$ and $s$ will always denote prime numbers,
except that in the the Appendix,
$s$ is a positive integer which measures the complexity of a 
system of linear forms.

Finally, we will be using the probabilistic method and will thus be working with finite probability spaces.
Generically we write $\PR$ for probability, and $\E$ for expectation. 
If a finite set $A$ is equipped with the uniform probability measure,
we write $\PR_{a\in A}$ and $\E_{a\in A}$ for the associated probability and
expectation.  Variables in boldface will denote random real-valued scalars,
while arrowed boldface symbols denote random vectors, e.g.
$\vect{\mathbf{a}}$.

We also use $\# A$ to denote the cardinality of $A$, and
for any positive real $z$, we let $[z] := \{ n \in \mathbf{N}: 1 \leq
n \leq z \}$ denote the set of natural numbers up to $z$.

%
\section{On arithmetic progressions consisting of primes}\label{sec:prime-progressions}
%

A key tool in the proof of Theorem ~\ref{mainthm-y} is an asymptotic formula
for counts of arithmetic progressions of primes. In fact, we shall be
interested in progressions of primes of length $r$ whose common
difference is $r!$ times a prime\footnote{One could replace $r!$ here
  if desired by the slightly smaller \emph{primorial} $P(r)$; 
as observed long ago by Lagrange and Waring \cite{dickson}, this primorial must divide the spacing of any sufficiently large arithmetic progression of primes of length $r$.  However, replacing $r!$ by $P(r)$ would lead to only a negligible savings in the estimates here.}, for positive integer values of
$r$. The key technical result we shall need is Lemma \ref{Q0} below. 
This is a relatively straightforward consequence of Lemma \ref{first} below, which relies on the work on linear equations in primes of the second and fourth authors and Ziegler.

We turn to the details.  Let $y$ be a sufficiently large quantity (which goes to infinity for the purposes of asymptotic notation), and let $x$ be a quantity that goes to infinity at a slightly slower rate than $y$; for sake of concreteness we will impose the hypotheses
\begin{equation}\label{xy}
x \sqrt{\log x} \leq y \leq x \log x.
\end{equation}
In fact the analysis in this section would apply under the slightly weaker hypotheses $y \log^{-O(1)} y \leq x \leq o(y)$, but we will stick with \eqref{xy} for sake of concreteness since this condition will certainly be satisfied when applying the results of this section to prove Theorem \ref{mainthm-y}.
From \eqref{xy} we see in particular that $\log y \asym \log x$, so we will use $\log x$ and $\log y$ more or less interchangeably in what follows.
Let $\PP$ denote the set of all primes in the interval $(x/2,x]$, and $\QQ$ denote the set of all primes in the interval $(x/4, y]$; thus from the prime number theorem we have
\begin{equation}\label{pq}
\# \PP \asym \frac{x}{2\log x}; \quad \# \QQ \asym \frac{y}{\log x}.
\end{equation}
In other words, $\PP$ and $\QQ$ both have density $\asym \frac{1}{\log x}$ inside $(x/2,x]$ and $(x/4,y]$ respectively.

Let $r \geq 1$ be a fixed natural number.  We define a relation $\rel$
between $\PP$ and $\QQ$  as follows: if $p \in \PP$ and $q \in \QQ$,
we write $p \rel q$ if the entire arithmetic progression $\{ q, q+r!
p, \dots, q+(r-1)r! p \}$ is contained inside $\QQ$.   
One may think of the $r$ relations $p \rel q-ir!p$ for $i=0,\dots,r-1$
as defining $r$ different (but closely related) bipartite graphs
between $\PP$ and $\QQ$.  Note that if $p \rel q$, then the residue
class $q\ \pmod{p}$ is guaranteed to contain at least $r$ primes from
$\QQ$, which is the main reason why we are interested in these
relations (particularly for somewhat large values of $r$).

For our main argument, we will be interested in the typical degrees of the bipartite graphs associated to the relations $p \rel q-ir!p$.  Specifically, we are interested\footnote{Actually, for technical reasons we will eventually replace the relation $\rel$ by slightly smaller relation $\relr$, which will in turn be randomly refined to an even smaller relation $\relra$; see below.} in the following questions for a given $0 \leq i \leq r-1$:

\begin{itemize}
\item[(i)] For a typical $p \in \PP$, how many $q \in \QQ$ are there such that $p \rel q-ir!p$?  (Note that the answer to this question does not depend on $i$.)
\item[(ii)]  For a typical $q \in \QQ$, how many $p \in \PP$ are there such that $p \rel q-ir!p$?
\end{itemize}

If $\PP$ and $\QQ$ were distributed randomly inside the intervals $(x/2,x]$ and $(x/4,y]$ respectively, with cardinalities given by \eqref{pq}, then standard probabilistic arguments (using for instance the Chernoff inequality) would suggest that the answer to question (i) is $\asym\frac{y}{\log^r x}$, while the answer to question (ii) is $\asym\frac{x}{2 \log^r x}$.  As it turns out, the local structure of the primes (for instance, the fact that all the elements of $\PP$ and $\QQ$ are coprime to $r!$) will bias the answers to each of these two questions; however (as one may expect from double counting considerations), they will be biased by exactly the same factor $\alpha_r$ (defined in \eqref{alpha-r-def} below), and the net effect of this bias will cancel itself out at the end of the proof of Theorem \ref{mainthm-y}.

One can predict the answers to Questions (i) and (ii) using the Hardy-Littlewood prime tuples conjecture \cite{hl}.  If we apply this conjecture (and ignore any issues as to how uniform the error term in that conjecture is with respect to various parameters), one soon arrives\footnote{See also Sections \ref{first-lemma-sec}, \ref{second-lemma-sec} for some closely related computations.} at the prediction that the answer to Question (i) should be $\asym\alpha_r \frac{y}{\log^r x}$ for all $p \in \PP$, and similarly the answer to Question (ii) should be $\asym \alpha_r \frac{x}{2\log^r x}$ for all $q \in \QQ$, where for the rest of the paper $\alpha_r$ will denote the singular series
\be\label{alpha-r-def}
 \alpha_r := \prod_{p \leq r} \left(\frac{p}{p-1}\right)^{r - 1} \prod_{p > r}\frac{(p-r)p^{r-1}}{(p-1)^r}.
\ee
The exact form of $\alpha_r$ is not important for our argument, so long as it is finite, positive, and does not depend on $x$ or $y$; but these claims are clear from \eqref{alpha-r-def} (note that the second factor $\frac{(p-r)p^{r-1}}{(p-1)^r}$ is non-zero and behaves asymptotically as $1+O(1/p^2)$).  As mentioned previously, this quantity will appear in two separate places in the proof of Theorem \ref{mainthm-y}, but these two occurrences will eventually cancel each other out. 

The Hardy-Littlewood conjecture is still out of reach of current technology.  Note that even the much weaker question as to whether the relation $p \rel q$ is satisfied for at least \emph{one} pair of $p$ and $q$ for any given $r$ is at least as hard as establishing that the primes contain arbitrarily long arithmetic progressions, which was only established by the second and fourth authors in \cite{gt-thm}.  However, for the argument used to prove Theorem \ref{mainthm-y}, it will suffice to be able to answer Question (i) for \emph{almost all} $p \in \PP$ rather than \emph{all} $p \in \PP$, and similarly for Question (ii).  In other words, we only need (a special case of) the Hardy-Littlewood prime conjecture ``on average''.  This is easier to establish; for instance, Balog \cite{balog} was able to use the circle method (or ``linear Fourier analysis'') to establish the prime tuples conjecture for ``most'' tuples in some sense.  The results in \cite{balog} are not strong enough for our applications, because of our need to consider arbitrarily long arithmetic progressions (which are well-known to not be amenable to linear Fourier-analytic methods for $r \geq 4$, see \cite{gowers-4}) rather than arbitrary prime tuples.  Instead we will use (a modification of) the more recent work of the second and fourth authors \cite{gt-linearprimes}.   More precisely, we claim the following bounds.

\begin{lem}\label{first}  Let $x,y,r,\PP,\QQ$, and $\rel$ be as above.   Let $0 \leq i \leq r-1$.
\begin{itemize}
\item[(i)] For all but $o(x/\log x)$ of the $p \in \PP$, we have the estimate
\[ \# \{ q \in \QQ: p \rel q-ir!p \} \asym  \alpha_r \frac{y}{\log^r x}.\]
\item[(ii)] For all but $o(y/\log x)$ of the $q \in \QQ$, we have
\[ \# \{ p \in \PP : p \rel q-ir!p \} \asym  \alpha_r  \frac{x}{2\log^r x}.\]
\item[(iii)]  For \emph{all} $p \in \PP$, we have the upper bounds
\[ \# \{ q \in \QQ: p \rel q-ir!p \} \ll \frac{y}{\log^r x}.\]
\item[(iv)]  For \emph{all} $q \in \QQ$, we have the upper bounds
\[ \# \{ p \in \PP : p \rel q-ir!p \} \ll \frac{x}{2\log^r x}.\]
\end{itemize}
\end{lem}

Parts (iii) and (iv) follow from standard sieve-theoretic methods (e.g. the Selberg sieve); we omit the proof here, referring the reader instead\footnote{One could also deduce these bounds from Proposition 6.4' in Appendix \ref{linear-primes-app}.} to \cite{HR} or \cite{FI}.  The more interesting bounds are (i) and (ii).
As stated above, these two claims are \emph{almost} relatively straightforward consequences of the main result of the paper \cite{gt-linearprimes} of the second and fourth authors. However, some modifications of that work are required to deal with the fact that $x$ and $y$ are of somewhat different sizes. In Section \ref{dickson-shifted-sec} below we state the modified version of the main result of \cite{gt-linearprimes} that we need, Theorem \ref{dickson-shifted}. The deductions of parts (i) and (ii) of Lemmas \ref{first} are
 rather similar to one another, and are given in Sections \ref{first-lemma-sec} and \ref{second-lemma-sec} respectively. Finally, a proof of Theorem \ref{dickson-shifted} can be obtained by modifying the arguments of \cite{gt-linearprimes} in quite a straightforward manner, but in a large number of places. We record these modifications in Appendix  \ref{linear-primes-app}.

As presently defined, it is possible for the bipartite graphs given by the $p \rel q-ir!p$ to overlap, thus it may happen that $p \rel q-ir!p$ and $p \rel q-jr!p$ for some $p \in \PP$, $q \in \QQ$, and $0 \leq i < j \leq r-1$.  For instance, this situation will occur if $\QQ$ has an arithmetic progression $q,q+r!p,\dots,q+r \times r!p$ of length $r+1$ with $p \in \PP$.  For technical reasons, such overlaps are undesirable for our applications.  However, these overlaps are rather rare and can be easily removed by the following simple device.  We define the modified relation $\relr$ between $\PP$ and $\QQ$ by declaring $p \relr q$ if the progression $\{q, q+r!p,\dots,q+(r-1)r!p\}$is contained inside $\QQ$, but $q+r \times r! p$ does \emph{not} lie in $\QQ$.  From construction we have the following basic fact:

\begin{lem}\label{disj} For any $p \in \PP$ and $q \in \QQ$ there is at most one $0 \leq i \leq r-1$ such that $p \relr q - ir!p$.
\end{lem}

We can then modify Lemma \ref{first} slightly by replacing the relation $\rel$ with its slightly perturbed version $\relr$:

\begin{lem}\label{Q0-a}  Let $x,y,r,\PP,\QQ$, and $\relr$ be as above.   Let $0 \leq i \leq r-1$.
\begin{itemize}
\item[(i)] For all but $o(x/\log x)$ of the $p \in \PP$, we have the estimate
\[ \# \{ q \in \QQ: p \relr q-ir!p \} \asym  \alpha_r \frac{y}{\log^r x}.\]
\item[(ii)] For all but $o(y/\log x)$ of the $q \in \QQ$, we have
\[ \# \{ p \in \PP : p \relr q-ir!p \} \asym  \alpha_r  \frac{x}{2\log^r x}.\]
\item[(iii)]  For \emph{all} $p \in \PP$, we have the upper bounds
\[ \# \{ q \in \QQ: p \relr q-ir!p \} \ll \frac{y}{\log^r x}.\]
\item[(iv)]  For \emph{all} $q \in \QQ$, we have the upper bounds
\[ \# \{ p \in \PP : p \relr q-ir!p \} \ll \frac{x}{2\log^r x}.\]
\end{itemize}
\end{lem}

\begin{proof}  Parts (iii) and (iv) are immediate from their counterparts in Lemma \ref{first}, since $\relr$ is a subrelation of $\rel$.  To prove (i), we simply observe from Lemma \ref{first}(iii) (with $r$ replaced by $r+1$) that
\[ \# \{ q \in \QQ: p \rel q-ir!p \hbox{ but } p \not \relr q-ir!p \} \ll \frac{y}{\log^{r+1} x},\]
and the claim then follows from Lemma \ref{first}(i) and the triangle inequality.  The claim (ii) is proven similarly.
\end{proof}

For technical reasons, it will be convenient to reformulate the main results of Lemma \ref{Q0-a} as follows.

\begin{lem}\label{Q0}  Let $x,y,r,\PP,\QQ$, and $\relr$ be as above.   Then there exist subsets $\PP_0$, $\QQ_0$ of $\PP, \QQ$ respectively with
\begin{equation}\label{pq0}
\# \PP_0 \asym \frac{x}{2\log x}; \quad \# \QQ_0 \asym \frac{y}{\log x},
\end{equation}
such that
\begin{equation}\label{q0-i}
 \# \{ q \in \QQ: p \relr q-ir!p \} \asym  \alpha_r \frac{y}{\log^r x}
\end{equation}
for \emph{all} $p \in \PP_0$ and $0 \leq i \leq r-1$, and similarly that
\begin{equation}\label{q0-ii}
 \# \{ p \in \PP_0 : p \relr q-ir!p \} \asym  \alpha_r  \frac{x}{2\log^r x}
\end{equation}
for \emph{all} $q \in \QQ_0$ and $0 \leq i \leq r-1$.
\end{lem}

\begin{proof}  From Lemma \ref{Q0-a}(i) we may already find a subset $\PP_0$ of the desired cardinality obeying \eqref{q0-i}.  If the $\PP_0$ in \eqref{q0-ii} were replaced by $\PP$, then a similar argument using Lemma \ref{Q0-a}(ii) (and taking the union bound for the exceptional sets for each $0 \leq i \leq r-1$) would give the remainder of the lemma.  To deal with the presence of $\PP_0$ in \eqref{q0-ii}, it thus suffices to show that
$$
 \# \{ p \in \PP \backslash \PP_0 : p \relr q-ir!p \} = o\left( \frac{x}{\log^r x} \right)$$
for all but $o(y/\log x)$ of the $q \in \QQ$.  By Markov's inequality, it suffices to show that
$$
 \# \{ (p,q) \in (\PP \backslash \PP_0) \times \QQ : p \relr q-ir!p \} = o\left( \frac{x}{\log^r x} \times \frac{y}{\log x}\right).$$
But this follows by summing Lemma \ref{Q0-a}(iii) for all $p \in \PP \backslash \PP_0$, since the set $\PP \backslash \PP_0$ has cardinality $o(x/\log x)$.
\end{proof}

\section{Main construction}
%

We now begin the proof of Theorem \ref{mainthm-y}.  It suffices to establish the following claim:

\begin{thm}[First reduction]\label{first-red} Let $r \geq 13$ be an integer.  Take $x$ to be sufficiently large depending on $r$ (and going to infinity for the purposes of asymptotic notation), and then define $y$ by the formula
\be\label{ydef}
y := \frac{r}{6\log r}\, \frac{x\log x\log_3 x}{(\log_2 x)^2}.
\ee
Then there exists a residue class $a_s \pmod s$ for each
prime $s\le x$, such that the union of these classes
contains every positive integer less than or equal to $y$.
\end{thm}

The numerical values of $13$ and $6$ in the above theorem are only of minor significance, and can be ignored for a first reading.

Observe that $x,y$ obey the condition \eqref{xy} from the previous section.  If Theorem \ref{first-red} holds, then in  terms of the quantity $Y(x)$ defined in the introduction, we have
$$ Y(x) \geq y$$
which by \eqref{ydef} will imply Theorem \ref{mainthm-y} by taking $r$ sufficiently large depending on $R$.

It remains to prove Theorem \ref{first-red}.  Set
\begin{equation}\label{zdef}
 z :=x^{\log_3 x/(3\log_2 x)},
\end{equation}
and partition the primes less than or equal to $x$ into the four disjoint classes
\begin{align*}
\cS_1 &:= \{ s \; \mbox{prime} : s\le \log x \text{ or } z < s \le x/4 \}\\
\cS_2 &:= \{ s\; \mbox{prime} : \log x < s \le z \}\\
\cS_3 &:= \PP = \{ s \; \mbox{prime}: x/2 < s \leq x\}\\
\cS_4 &:= \{s \;\mbox{prime}: x/4 < s \leq x/2\}.
\end{align*}

We are going to sieve $[y]$ in four stages by removing at most one congruence class $a_s \pmod{s}$ for each prime $s \in S_i$, $i = 1,2,3,4$. If we can do this in such a way that nothing is left at the end, we shall have achieved our goal.

We first dispose of the final sieving process (involving $\cS_4$), as it is rather trivial.  Namely, we reduce Theorem \ref{first-red} to

\begin{thm}[Second reduction]\label{second-red}  Let $r,x,y$ be as in
  Theorem \ref{first-red}, and let $\cS_1,\cS_2,\cS_3$ be as above.
  Then there exists a residue class $a_s \pmod s$ for each $s \in
  \cS_1 \cup \cS_2 \cup \cS_3$, such that the union of these classes
  contains all but at most $(\frac{1}{5} + o(1)) \frac{x}{\log x}$ of the positive integers less than or equal to $y$.
\end{thm}

Indeed, if the $a_s \pmod s$ for $s \in \cS_1 \cup \cS_2 \cup \cS_3$ are as in Theorem \ref{second-red}, then from the prime number theorem, the number of integers less than $y$ that have not already been covered by a residue class is smaller than the number of primes in $\cS_4$.  Thus, we may eliminate each of these surviving integers using a residue class $a_s \pmod s$ from a different element $s$ from $\cS_4$ (and selecting residue classes arbitrarily for any $s \in \cS_4$ that are left over), and Theorem \ref{first-red} follows.

It remains to prove Theorem \ref{second-red}.  For this, we perform the first sieving process (using up the primes from $\cS_1$) and reduce to

\begin{thm}[Third reduction]\label{third-red}  Let $r,x,y,\cS_2,\cS_3$
  be as in Theorem \ref{second-red}, and (as in the previous section)
  let $\QQ$ denote the primes in the range $(x/4,y]$.  Then there
    exists a residue class $a_s \pmod s$ for each $s \in \cS_2 \cup
    \cS_3$, such that the union of these classes contains all but at most $(\frac{1}{5} + o(1)) \frac{x}{\log x}$ of the elements of $\QQ$.
\end{thm}

\begin{proof}[Proof of Theorem \ref{second-red} assuming Theorem \ref{third-red}]  We take $a_s := 0$ for all $s\in \cS_1$.
Write $\cR \subset [y]$ for the residual set of elements
which survive this first sieving, that is to say $\cR$ consists of all numbers in $[y]$ that are not divisible by any prime $s$ in $\cS_1$. Taking into account that
$(x/4)\log x>y$ from \eqref{ydef}, we conclude that 
$$\cR = \QQ \cup \cR^{\err},$$
where $\cR^{\err}$ contains only $z$-smooth
  numbers, that is to say numbers in $[y]$ all of
  whose prime factors are at most $z$.

Let $u$ denote the quantity
$$ u := \frac{\log y}{\log z},$$
so from \eqref{zdef} one has $u \asym 3 \frac{\log_2 x}{\log_3 x}$.  
By standard counts for smooth numbers (e.g. de Bruijn's theorem \cite{deB}),
\begin{align*}
\# \cR^{\err}  &\ll y e^{-u\log u + O( u \log\log(u+2) ) } \\
&= \frac{y}{\log^{3+o(1)} x} \\
&= \frac{x}{\log^{2+o(1)} x} \\
&= o( x/\log x).
\end{align*}
Thus the contribution of $\cR^{\err}$ may be absorbed into the exceptional set in Theorem \ref{second-red}, and this theorem is now immediate from Theorem \ref{third-red}.
\end{proof}

\begin{remark} One can replace the appeal to de Bruijn's theorem here by the simpler bounds of Rankin \cite[Lemma II]{R1}, if one makes the very minor change of increasing the $3$ in the denominator of \eqref{zdef} to $4$, and to similarly increase the $6$ in \eqref{ydef} to $8$. 
\end{remark}

It remains to establish Theorem \ref{third-red}.
Recall from \eqref{pq} that $\QQ$ has cardinality $\asym y/\log x$.  This is significantly larger than the error term of $(\frac{1}{5}+o(1)) \frac{x}{\log x}$ permitted in Theorem \ref{third-red}; our sieving process has to reduce the size of $\QQ$ by a factor comparable to $y/x$.  The purpose of the second sieving, by congruences $\mathbf{a}_s \pmod{s}$ with $s \in \cS_2$, is to achieve almost all of this size reduction. Our choice of the $\mathbf{a}_s$ for $s \in \cS_2$ will be completely random (which is why we are using the boldface font here): that is, for each prime $s \in \cS_2$ we select $\mathbf{a}_s$ uniformly at random from $\{0,1,\dots, s-1\}$, and these choices are independent for different values of $s$. Write $\vect{\mathbf{a}}$ for the random vector $(\mathbf{a}_s)_{s \in \cS_2}$.

Observe that if $n$ is any integer (not depending on $\vect{\mathbf{a}}$), then the probability that $n$ lies outside of all of the $\mathbf{a}_s \pmod{s}$ is exactly equal to
$$ \gamma := \prod_{s \in S_2} \left(1 - \frac{1}{s}\right).$$
This quantity will be an important normalizing factor in the arguments that follow.
From Mertens' theorem and \eqref{ydef}, \eqref{zdef} we see that
\begin{equation}\label{gamma-form}
\gamma \asym \frac{\log_2 x}{\log z} \asym \frac{3 (\log_2 x)^2}{\log x \log_3 x} \asym  \frac{r}{2\log r} \frac{x}{y}.
\end{equation}

Write $\QQ(\vect{\mathbf{a}})$ for the (random) residual set of primes $q$ in $\QQ$ that do not lie in any of the congruence classes $\mathbf{a}_s \pmod{s}$ for $s \in \cS_2$.  We will in fact focus primarily on the slightly smaller set
$$ \QQ_0(\vect{\mathbf{a}}) := \QQ(\vect{\mathbf{a}}) \cap \QQ_0$$
where $\QQ_0$ is the subset of $\QQ$ constructed in Lemma \ref{Q0}. 
From linearity of expectation we see that
\begin{equation}\label{eqa}
\E \# \QQ(\vect{\mathbf{a}}) = \gamma \# \QQ
\end{equation}
and thus from \eqref{gamma-form}, \eqref{pq}
\begin{equation}\label{qa}
 \E \# \QQ(\vect{\mathbf{a}}) \asym \frac{r}{2\log r} \frac{x}{\log x}.
\end{equation}
Similarly, from Lemma \ref{Q0} we have
$$ \# (\QQ \backslash \QQ_0) = o\left( \frac{y}{\log x} \right)$$
and thus from linearity of expectation and \eqref{gamma-form} we have
$$ \E \# (\QQ(\vect{\mathbf{a}}) \backslash \QQ_0(\vect{\mathbf{a}}) ) = o\left( \gamma \frac{y}{\log x} \right) = o\left( \frac{x}{\log x} \right).$$
In particular, from Markov's inequality we have
\begin{equation}\label{qq-qq0}
\# ( \QQ(\vect{\mathbf{a}}) \backslash \QQ_0(\vect{\mathbf{a}}) ) = o\left( \gamma \frac{y}{\log x} \right) = o\left( \frac{x}{\log x} \right)
\end{equation}
with probability $1-o(1)$.

\begin{table}
\begin{tabular}{|l|l|l|}
\hline
Set & Description & Expected cardinality \\
\hline
$\PP$ & Primes in $(x/2,x]$ & $\asym \frac{x}{2\log x}$ \\
$\PP_0$ & Primes in $\PP$ connected to the expected \# of primes in $\QQ$ & $\asym \frac{x}{2\log x}$\\ 
$\PP_1(\vect{\mathbf{a}})$ & Primes in $\PP_0$ connected to the expected \# of primes in $\QQ(\vect{\mathbf{a}})$ & $\asym \frac{x}{2\log x}$\\
$\PP_1(\vect{\mathbf{a}},q;i)$ & Primes in $\PP_1(\vect{\mathbf{a}})$ $i$-connected to a given prime $q \in \QQ_1(\vect{\mathbf{a}})$ & $\asym \gamma^{r-1} \alpha_r \frac{x}{2\log^r x}$ \\
\hline
$\QQ$ & Primes in $(x/4,y]$ & $\asym \frac{y}{\log x}$ \\
$\QQ_0$ & Primes in $\QQ$ connected to the expected \# of primes in $\PP_0$ & $\asym \frac{y}{\log x}$ \\
$\QQ(\vect{\mathbf{a}})$ & Randomly refined subset of $\QQ$ & $\asym \frac{r}{2\log r} \frac{x}{\log x}$ \\
$\QQ(\vect{\mathbf{a}}, p)$ & Primes in $\QQ(\vect{\mathbf{a}})$ connected to a given prime $p \in \PP_1(\vect{\mathbf{a}})$ & $\asym \gamma^r\alpha_r\frac y{\log^r x}$ \\
$\QQ_0(\vect{\mathbf{a}})$ & Intersection of $\QQ(\vect{\mathbf{a}})$ with $\QQ_0$ & $\asym \frac{r}{2\log r} \frac{x}{\log x}$ \\
$\QQ_1(\vect{\mathbf{a}})$ & Primes in $\QQ_0(\vect{\mathbf{a}})$ connected to the expected \# of primes in $\PP_1(\vect{\mathbf{a}})$ & $\asym \frac{r}{2\log r} \frac{x}{\log x}$ \\
$\QQ_1(\vect{\mathbf{a}},\vect{\mathbf{q}})$ & Randomly refined subset of $\QQ_1(\vect{\mathbf{a}})$ & $\asym \frac{1}{2\log r} \frac{x}{\log x}$\\
\hline
\end{tabular}
\caption{A brief description of the various $\PP$ and $\QQ$-type sets used in the construction, and their expected size.  Roughly speaking, the congruence classes from $\cS_1$ are used to cut down $[y]$ to approximately $\QQ$, the congruence classes from $\cS_2$ are used to cut $\QQ$ down to approximately $\QQ_0(\vect{\mathbf{a}})$, the congruence classes from $\cS_3 = \PP$ are used to cut $\QQ_0(\vect{\mathbf{a}})$ down to approximately $\QQ_1(\vect{\mathbf{a}}, \vect{\mathbf{q}})$, and the congruence classes in $\cS_4$ are used to cover all surviving elements from previous sieving.}
\end{table}

We have an analogous concentration bound for $\# \QQ(\vect{\mathbf{a}})$:

\begin{lem}\label{QQ1_normal}  With probability $1-o(1)$, we have
$$ \# \QQ(\vect{\mathbf{a}}) \asym \frac{r}{2\log r} \frac{x}{\log x} \asym \gamma \frac{y}{\log x}.$$
In particular, from \eqref{qq-qq0} we also have
$$ \# \QQ_0(\vect{\mathbf{a}}) \asym \frac{r}{2\log r} \frac{x}{\log x} \asym \gamma \frac{y}{\log x}$$
with probability $1-o(1)$.
\end{lem}

This lemma is proven by a routine application of the second moment method; we defer that proof to Section \ref{sec:prob}.  It will now suffice to show

\begin{thm}[Fourth reduction]\label{fourth-red}  Let
  $x,y,r,\vect{\mathbf{a}},\PP_0,\QQ_0(\vect{\mathbf{a}})$ be as
  above, and let $\eps > 0$ be a quantity going to zero arbitrarily
  slowly as $x \to \infty$, thus $\eps = o(1)$.  Then with probability
  at least $\eps$ in the random choice of $\vect{\mathbf{a}}$, 
we may find a length $r$  arithmetic progression $\{ q_p + i r! p: 0
\leq i \leq r-1\}$ for each $p \in \PP_0$, such that the union of
these progressions contains all but at most $(\frac{1}{5} + o(1)) \frac{x}{\log x}$ of the elements of $\QQ_0(\vect{\mathbf{a}})$.  (The $o(1)$ decay in the conclusion may depend on $\eps$.)
\end{thm}

Indeed, from this theorem (and taking $\eps$ going to zero sufficiently slowly) we may find $\vect{\mathbf{a}}$ such that the conclusions of this theorem hold simultaneously with \eqref{qq-qq0}, and by combining the residue classes from $\vect{\mathbf{a}}$ with the residue classes $q_p \pmod{p}$ for $p \in \PP_0$ from Theorem \ref{fourth-red} (and selecting residue classes arbitrarily for $p \in \PP \backslash \PP_0$), we obtain Theorem \ref{third-red}.

It remains to establish Theorem \ref{fourth-red}.  Note now (from Lemma \ref{QQ1_normal}) that we only need to reduce the size of the surviving set $\QQ_0(\vect{\mathbf{a}})$ through sieving by a constant factor (comparable to $\frac{r}{\log r}$), rather than by a factor like $y/x$ that goes to infinity as $x \to \infty$.

Recall from the previous section that we had the relation $\relr$ between $\PP$ and $\QQ$.  We now refine this relation to a (random) relation between $\PP_0$ and $\QQ(\vect{\mathbf{a}})$ as follows.  If $p \in \PP_0$ and $q \in \QQ(\vect{\mathbf{a}})$, we write $p \relra q$ if $p \relr q$ and if the arithmetic progression $\{ q, q+r!p, \dots, q+(r-1)r!p \}$ is contained in $\QQ(\vect{\mathbf{a}})$ (i.e. the entire progression survives the second sieving process).  

Intuitively, if $p \in \PP_0$ and $q \in \QQ$ are such that $p \relr q$, we expect $p \relra q$ to occur with probability close to $\gamma^r$.  The following lemma makes this intuition precise (compare with Lemma \ref{Q0}):

\begin{lem}\label{sieveunited} 
Let $\eps > 0$ be a quantity going to zero arbitrarily slowly as $x \to \infty$.
Then with probability at least $\eps$, we can find (random) subsets
$\PP_1(\vect{\mathbf{a}})$  of $\PP_0$ and $\QQ_1(\vect{\mathbf{a}})$
of $\QQ_0(\vect{\mathbf{a}})$ obeying the cardinality
bounds
\begin{equation}\label{ppqq1}
\# \PP_1(\vect{\mathbf{a}}) \asym \frac{x}{2\log x}; \quad \# \QQ_1(\vect{\mathbf{a}}) \asym \# \QQ_0(\vect{\mathbf{a}}) \asym \frac{r}{2\log r} \frac{x}{\log x},
\end{equation}
such that
$$ \# \{ q \in \QQ(\vect{\mathbf{a}}): p \relra q - ir! p \} \asym \gamma^r\alpha_r\frac y{\log^r x}$$
for \emph{all} $p \in \PP_1(\vect{\mathbf{a}})$ and $0 \leq i \leq r-1$, and such that
$$ \# \{ p \in \PP_1(\vect{\mathbf{a}}):  p \relra q - ir! p \} \asym \gamma^{r-1} \alpha_r \frac{x}{2\log^r x} $$
for \emph{all} $q \in \QQ_1(\vect{\mathbf{a}})$ and $0 \leq i \leq
r-1$.  (The implied $o(1)$ errors in the $\asym$ notation may depend on $\eps$.)
\end{lem}

This lemma is also proven by an application of the second moment method; we defer this proof also to Section \ref{sec:prob}.

We are now ready to perform the third sieving process.
 Let us fix any $\vect{\mathbf{a}}$ obeying the
properties in Lemma \ref{sieveunited}, and let
$\PP_1(\vect{\mathbf{a}})$ and $\QQ_1(\vect{\mathbf{a}})$ be as in
that lemma.  Since $\vect{\mathbf{a}}$ has the desired properties 
with probability at least $\eps$, in order to establish
Theorem \ref{fourth-red} (and thus Theorem \ref{mainthm-y} and Theorem
\ref{mainthm}), it suffices to show that for every such
 $\vect{\mathbf{a}}$, there is a choice of residue classes $q_p$ for 
$p\in \PP_0$ satisfying the required union property for Theorem
\ref{fourth-red}.

For each $p \in \PP_1(\vect{\mathbf{a}})$, we select $\mathbf{q}_p$ uniformly at random from the set
\begin{equation}\label{Bp-def}
\QQ(\vect{\mathbf{a}}, p) := \{ q \in \QQ(\vect{\mathbf{a}}): p \relra q \}, 
\end{equation}
with the $\mathbf{q}_p$ for $p \in \PP_1(\vect{\mathbf{a}})$ being chosen independently (after $\vect{\mathbf{a}}$ has been fixed); note from Lemma \ref{sieveunited} that
\begin{equation}\label{Bp-size}
\# \QQ(\vect{\mathbf{a}}, p) \asym \gamma^r\alpha_r\frac y{\log^r x}
\end{equation}
for all $p \in \PP_1(\vect{\mathbf{a}})$.  We write
$\vect{\mathbf{q}}$ for the random tuple $(\mathbf{q}_p)_{p \in
  \PP_1(\vect{\mathbf{a}}) }$, and for brevity write
 $\PR_\vect{\mathbf{q}}$ and $\E_\vect{\mathbf{q}}$ for the associated 
probability and expectation with respect to this random tuple 
(where $\vect{\mathbf{a}}$ is now fixed).
  Let $\QQ_1(\vect{\mathbf{a}}, \vect{\mathbf{q}})$ denote the elements of $\QQ_1(\vect{\mathbf{a}})$ that are not covered by any of the arithmetic progressions $\{ \mathbf{q}_p + i r! p: 0 \leq i \leq r-1\}$ for each $p \in \PP_1(\vect{\mathbf{a}})$.  We claim that
\begin{equation}\label{mbq}
 \E_\vect{\mathbf{q}}  \# \QQ_1(\vect{\mathbf{a}}, \vect{\mathbf{q}}) \leq  \frac{x}{5\log x}.
\end{equation}
This implies (for each fixed choice of $\vect{\mathbf{a}}$) the existence of a vector $\vect{q}$ with
$$  \# \QQ_1(\vect{\mathbf{a}}, \vect{q}) \leq  \frac{x}{5\log x};$$
since $\# (\QQ_0(\vect{\mathbf{a}}) \backslash
\QQ_1(\vect{\mathbf{a}})) = o(x/\log x)$ from \eqref{ppqq1}, Theorem
\ref{fourth-red} follows (upon choosing 
$q_p$ as the $p$ component of $\vect{q}$ for  $p\in\PP_1(\vect{\mathbf{a}})$ and  $q_p$ arbitrarily for $p \in \PP_0 \backslash \PP_1(\vect{\mathbf{a}})$).

It remains to prove \eqref{mbq}.  We will shortly show that
\begin{equation}\label{pqr}
 \PR_\vect{\mathbf{q}} ( q \in \QQ_1(\vect{\mathbf{a}}, \vect{\mathbf{q}}) ) \leq \frac{1+o(1)}{r}
\end{equation}
for each $q \in \QQ_1(\vect{\mathbf{a}})$.  Assuming
this bound, then from \eqref{ppqq1} and linearity of expectation we
have
$$
 \E_\vect{\mathbf{q}}  \# \QQ_1(\vect{\mathbf{a}}, \vect{\mathbf{q}}) \leq \frac{1+o(1)}{r} \frac{r}{2\log r} \frac{x}{\log x}$$
which gives \eqref{mbq} as desired for $r \geq 13$.

It remains to prove \eqref{pqr}.  Fix $q \in \QQ_1(\vect{\mathbf{a}})$, and consider the sets
$$ \PP_1(\vect{\mathbf{a}},q; i) := \{ p \in \PP_1(\vect{\mathbf{a}}):  p \relra q - ir! p \}$$
for $i=0,\dots,r-1$.  From Lemma \ref{disj}, these sets are disjoint; from Lemma \ref{sieveunited}, these sets each have cardinality $(1+o(1)) \gamma^{r-1} \alpha_r \frac{x}{2\log^r x}$.

Suppose that $0 \leq i \leq r-1$ and $p \in \PP_1(\vect{\mathbf{a}},q;
i)$.  Then $q-ir! p \in \QQ(\vect{\mathbf{a}},p)$ by \eqref{Bp-def},
and the probability that $\mathbf{q}_p=q-ir!p$ is equal to
$$  \frac{1}{\# \QQ(\vect{\mathbf{a}}, p)} = \frac{1+o(1)}{\gamma^r\alpha_r\frac y{\log^r x} }$$
thanks to \eqref{Bp-size}. By independence, the probability that $\mathbf{q}_p \neq q-ir!p$ for all $0 \leq i \leq r-1$ and $p \in \PP_1(\vect{\mathbf{a}},q;i)$ (which is a necessary condition for $q$ to end up in $\QQ_1(\vect{\mathbf{a}}, \vect{\mathbf{q}})$) is thus
\begin{align*}
 \prod_{i=0}^{r-1} \; \prod_{p \in \PP_1(\vect{\mathbf{a}},q; i)}\left( 1 - \frac{1+o(1)}{\gamma^{r} \alpha_r \frac{y}{\log^r x} } \right)
&= \exp\left( - \frac{1+o(1)}{\gamma^{r} \alpha_r \frac{y}{\log^r x} } r (1+o(1)) \gamma^{r-1}\alpha_r\frac x{2\log^r x} \right) \\
&= \exp\left( - (1+o(1)) \frac{rx}{2y\gamma} \right ) \\
&= \exp( - (1+o(1)) \log r )
\end{align*}
by \eqref{gamma-form}.  The claim \eqref{pqr} follows.

\section{Probability estimates}\label{sec:prob}
%

In this section we establish the results left unproven in the last
section, namely Lemmas \ref{QQ1_normal} and \ref{sieveunited}.  Our primary tool here will be the second moment method.  Throughout, 
the probabilistic quantities we write are
all with respect to the random choice of the vector
$\vect{\mathbf{a}}=(\mathbf{a}_s)_{s\in S_2}$.
In several of these proofs we will make use of the quantities $\gamma_i$ defined by
\begin{equation}\label{gammai-def}
\gamma_i:=\prod_{s\in \cS_2} \left(1-\frac is\right)
\end{equation}
for $i = 1,\dots, 2r$. Note that $\gamma_1 = \gamma$ in the notation of the previous section.

\begin{lem}\label{gam-bounds}
We have $\gamma_i \asym \gamma^i$, uniformly for all $1 \leq i \leq 2r$.

\end{lem}
\begin{proof} We have, uniformly for $1 \leq i \leq 2r$,
\[
\gamma_i=\gamma^{i}\prod_{s\in \cS_2} \left(1-\frac is\right)
\left(1-\frac 1s\right)^{-i}
=\g^i\prod_{s\in \cS_2}\big(1+O(s^{-2})\big)
=\g^i(1+O(1/\log x)),
\]
using the fact that all primes $s \in \cS_2$ are $> \log x$.
\end{proof}

\subsection{Proof of Lemma \ref{QQ1_normal}}

To prove Lemma \ref{QQ1_normal} we use the second moment method.  Indeed, from Chebyshev's inequality it will suffice to prove the asymptotics
\begin{equation}\label{mean-1}
\E \# \QQ(\vect{\mathbf{a}}) \asym \gamma \frac{y}{\log x}
\end{equation}
and
\begin{equation}\label{mean-2}
\E (\# \QQ(\vect{\mathbf{a}}))^2 \asym \left(\gamma \frac{y}{\log x}\right)^2.
\end{equation}
The claim \eqref{mean-1} is just \eqref{qa}, so we turn to \eqref{mean-2}. 
The left-hand side of \eqref{mean-2} may be written as
$$  \sum_{q_1, q_2 \in \QQ} \PR( q_1,q_2 \in \QQ(\vect{\mathbf{a}}) ).$$
The diagonal contribution $q_1=q_2$ is clearly negligible (it is crudely bounded by $\# \QQ$, which is much smaller than $\left(\gamma \frac{y}{\log x}\right)^2$), so by \eqref{pq} it suffices to show that
$$ \PR( q_1,q_2 \in \QQ(\vect{\mathbf{a}}) ) \asym \gamma^2$$
for any \emph{distinct} $q_1, q_2 \in \QQ$.

Fix any such $q_1,q_2$.  Observe that for each $s \in \cS_2$, the probability that $q_1$ and $q_2$ simultaneously avoid $\mathbf{a}_s \pmod{s}$ is equal to $1-\frac{2}{s}$ if $s$ does not divide $q_2-q_1$, and $1-\frac{1}{s}$ otherwise.  In the latter case, we crudely write $1-\frac{1}{s}$ as $(1 + O(\frac{1}{\log x})) (1-\frac{2}{s})$.  Since $q_2-q_1 = O(y)$ and all the primes in $\cS_2$ are at least $\log x$, we see that there are at most $O(\frac{\log y}{\log \log x}) = o(\log x)$ primes $s$ that divide $q_2-q_1$.  We conclude that
$$ \PR( q_1,q_2 \in \QQ(\vect{\mathbf{a}}) ) = \left(1 +
O\pfrac{1}{\log x}\right)^{o(\log x)} \prod_{s \in \cS_2}
\left(1-\frac{2}{s}\right) \asym \gamma_2,$$
and the claim now follows from Lemma \ref{gam-bounds}.

\subsection{A preliminary lemma}

In order to establish Lemma \ref{sieveunited}, we will first need the following preliminary result in this direction.

\begin{lem}\label{QQ1p-normal}  The following two claims hold with probability $1-o(1)$ (in the random choice of $\vect{\mathbf{a}}$), and for any $0 \leq i \leq r-1$.
\begin{itemize}
\item[(i)]  One has
\begin{equation}\label{po}
\# \{ q \in \QQ(\vect{\mathbf{a}}): p \relra q - ir! p \} \asym \gamma^r \alpha_r \frac{y}{\log^r x} \asym \gamma^r \# \{ q \in \QQ: p \relr q-ir! p \}
\end{equation}
for all but $o(x/\log x)$ values of $p \in \PP_0$.
\item[(ii)] One has
\begin{equation}\label{qo}
 \# \{ p \in \PP_0: p \relra q - ir! p \} \asym \gamma^{r-1} \alpha_r \frac{x}{2\log^r x} \asym \gamma^{r-1} \# \{ p \in \PP_0: p \relr q-ir! p \}
\end{equation}
for all but $o(x/\log x)$ values of $q \in \QQ_0(\vect{\mathbf{a}})$.
\end{itemize}
\end{lem}

We begin with the proof of Lemma \ref{QQ1p-normal}(i), which goes along very similar lines to that of the previous lemma.  As the quantities here do not depend on $i$, we may take $i=0$.  The second part of \eqref{po} follows from \eqref{q0-i}, so it suffices to show that with probability $1-o(1)$, we have
\begin{equation}\label{qc}
\# \{ q \in \QQ(\vect{\mathbf{a}}): p \relra q \} \asym \gamma^r \alpha_r \frac{y}{\log^r x}
\end{equation}
for all but $o(x/\log x)$ values of $p \in \PP_0$.  By Markov's inequality and \eqref{pq0}, it suffices to show that for each $p \in \PP_0$, we have the event \eqref{qc} with probability $1-o(1)$. 

Fix $p \in \PP_0$. By Chebyshev's inequality, it suffices to show that
\begin{equation*}
 \E \# \{ q \in \QQ(\vect{\mathbf{a}}): p \relra q \} \asym \gamma^r \alpha_r \frac{y}{\log^r x}
\end{equation*}
and
\begin{equation*}
 \E \big(\# \{ q \in \QQ(\vect{\mathbf{a}}): p \relra q \}\big)^2 \asym \left(\gamma^{r} \alpha_r \frac{y}{\log^r x}\right)^2.
\end{equation*} 
By \eqref{q0-i}, Lemma \ref{gam-bounds}, and linearity of expectation, it thus suffices to show that
\begin{equation}\label{iso-1}
\PR( q, q+r!p, \dots, q+(r-1)r!p \in \QQ(\vect{\mathbf{a}}) ) \asym \gamma_r
\end{equation}
for all $q \in \QQ$ with $p \relr q$, and similarly that
\begin{equation}\label{iso-2}
\PR( q_1, q_1+r!p, \dots, q_1+(r-1)r!p,q_2,q_2+r!p,\dots,q_2+(r-1)r!p \in \QQ(\vect{\mathbf{a}}) ) \asym \gamma_{2r}
\end{equation}
for any distinct $q_1,q_2 \in \QQ$ with $p \relr q_1,p \relr q_2$.

We begin with \eqref{iso-1}.  For any $s \in \cS_2$, the probability that $q,q+r!p,\dots,q+(r-1)r!p$ simultaneously avoid $\mathbf{a}_s \pmod{s}$ is equal to $1 - \frac{r}{s}$ (note that $s$ is coprime to $r!p$).  So \eqref{iso-1} then follows (with exact equality) from \eqref{gammai-def} and independence.

Now we turn to \eqref{iso-2}.  For any $s \in \cS_2$, the probability
that $q_1, q_1+r!p, \dots,
q_1+(r-1)r!p,q_2,q_2+r!p,\dots,q_2+(r-1)r!p$ simultaneously avoid
$\mathbf{a}_s \pmod{s}$ is usually equal to $1-\frac{2r}{s}$; the
exceptions arise when $s$ divides $q_2-q_1 + i r! p$ for some $-r \leq
i \leq r$, in which case the probability is instead $(1 +
O(\frac{1}{\log x})) (1-\frac{2r}{s})$.  But by arguing as in the
proof of Lemma \ref{QQ1_normal}, the number of exceptional $s$ is
$o(\log x)$.  Multiplying all the independent probabilities together,
we obtain the claim \eqref{iso-2}.  This concludes the proof of Lemma \ref{QQ1p-normal}(i).

Now we prove Lemma \ref{QQ1p-normal}(ii).  Again, the second part of
\eqref{qo} follows from \eqref{q0-ii}.  For the first part, it
suffices (by Lemma \ref{QQ1_normal} and \eqref{gamma-form}) to show that with probability $1-o(1)$, one has
$$ \sum_{q \in \QQ_0(\vect{\mathbf{a}})} \left|\# \{ p \in \PP_0: p \relra q - ir! p \} - \gamma^{r-1} \alpha_r \frac{x}{2\log^r x}\right|^2 = o\left( \gamma \frac{y}{\log x} \(\gamma^{r-1} \frac{x}{\log^r x}\)^2 \right).$$
By Markov's inequality, it suffices to show that
$$ \E \sum_{q \in \QQ_0(\vect{\mathbf{a}})} \left|\# \{ p \in \PP_0: p \relra q - ir! p \} - \gamma^{r-1} \alpha_r \frac{x}{2\log^r x}\right|^2 = o\left( \gamma \frac{y}{\log x} \(\gamma^{r-1} \frac{x}{\log^r x}\)^2 \right).$$
Expanding out the square, it suffices to show the estimate
\begin{equation}\label{est-1}
\E \sum_{q \in \QQ_0(\vect{\mathbf{a}})} \(\# \{ p \in \PP_0: p \relra
q - ir! p \}\)^b  \asym  \gamma \frac{y}{\log x} \(\gamma^{r-1} \alpha_r
\frac{x}{2\log^r x}\)^b
\end{equation}
for $b=0,1,2$.

The $b=0$ case of \eqref{est-1} follows from \eqref{pq} and \eqref{eqa}.  For the $b=1,2$ cases, observe from Lemma \ref{Q0} that
$$
\sum_{q \in \QQ_0} (\# \{ p \in \PP_0: p \relr q - ir! p \})^b \asym \left(\alpha_r \frac{x}{2\log^r x}\right)^b \frac{y}{\log x}.
$$
By linearity of expectation, it thus suffices to show that
\begin{equation}\label{iso-1a}
\PR( q-ir!p, q+(1-i)r!p,\dots, q+(r-1-i)r! p \in \QQ_0(\vect{\mathbf{a}}) ) \asym \gamma^r
\end{equation}
whenever $p \in \PP_0$, $q \in \QQ_0$ with $p \relr q-ir! p$, and
\begin{equation}\label{iso-2a}
\PR( q+jr!p_k \in \QQ_0(\vect{\mathbf{a}}) \mbox{ for all }
j=-i,1-i,\ldots,r-1-i \mbox{ and } k=1,2 ) \asym \gamma^{2r-1}
\end{equation}
whenever $p_1, p_2 \in \PP_0$, $q \in \QQ_0$ with $p_1 \relr q-ir!
p_1$, $p_2 \relr q-ir! p_2$, and $p_1 \neq p_2$ (the total
contribution of the diagonal $p_1=p_2$ is easily seen to be
negligible). 

We begin with the proof of \eqref{iso-1a}.  For any $s \in \cS_2$, the probability that the progression $q-ir!p, q+(1-i)r!p,\dots, q+(r-1-i)r! p$ avoids $\mathbf{a}_s \pmod{s}$ is equal to $1-\frac{r}{s}$ (since $s$ is coprime to $r!p$), and so by \eqref{gammai-def} and independence the left-hand side of \eqref{iso-1a} is precisely $\gamma_r$.  The claim now follows from Lemma \ref{gam-bounds}.

Now we prove \eqref{iso-2a}.  For any $s \in \cS_2$, the probability that the intersecting progressions $q-ir!p_1, q+(1-i)r!p_1,\dots, q+(r-1-i)r! p_1$ and $q-ir!p_2, q+(1-i)r!p_2,\dots, q+(r-1-i)r! p_2$ avoid $s$ is usually $1 - \frac{2r-1}{s}$ (note that $q$ is a common value of the two arithmetic progressions).  The exceptions occur when $s$ divides $jp_1+kp_2$ for some $-r \leq j,k \leq r$ that are not both zero, but by arguing as before we see that the number of such exceptions is $o(\log x)$, and the probability in these cases is $(1 + O(\frac{1}{\log x})) (1 - \frac{2r-1}{s})$.  Thus by independence, the left-hand of \eqref{iso-2a} is $\asym \gamma_{2r-1}$, and the claim follows from Lemma \ref{gam-bounds}.  The proof of Lemma \ref{QQ1p-normal} is now complete.

\subsection{Proof of Lemma \ref{sieveunited}}

Suppose that $\eps>0$ goes to zero as $x \to \infty$ sufficiently
slowly.

Let $\PP_1(\vect{\mathbf{a}})$ be the set of $p \in \PP_0$ obeying \eqref{po} for all $0 \leq i \leq r-1$ (actually the choice of $i$ is irrelevant here), then from Lemma \ref{QQ1p-normal}(i) and \eqref{pq0} we have that with probability at least $1-\eps$ we have
\begin{equation}\label{pp1}
 \# \PP_1(\vect{\mathbf{a}}) \asym \frac{x}{2\log x}
\end{equation}
as required.  From Lemma \ref{QQ1_normal} we also have $ \# \QQ_0(\vect{\mathbf{a}}) \asym \frac{r}{2\log r} \frac{x}{\log x}$ with probability at least $1-\eps$ as required.  To finish the proof of the lemma, it suffices in view of Lemma \ref{QQ1p-normal}(ii) to show that with probability at least $3\eps$, one has
$$ 
\#  \{ p \in \PP_0 \backslash \PP_1(\vect{\mathbf{a}}):  p \relra q - ir! p \} = o( \gamma^{r-1} x / \log^r x )$$
for all but $o(x/\log x)$ values of $q \in \QQ_0(\vect{\mathbf{a}})$, and any $0 \leq i \leq r-1$.  

We use a double counting argument.  It clearly suffices to show with probability at least $3\eps$ that
$$ 
\#  \{ (p,q) \in (\PP_0 \backslash \PP_1(\vect{\mathbf{a}})) \times \QQ_0(\vect{\mathbf{a}}):  p \relra q - ir! p \} = o\left( \gamma^{r-1} \frac{x}{\log^r x} \times \frac{x}{\log x} \right)
$$
for all $0 \leq i \leq r-1$.  Actually, the left-hand side does not
depend on $i$ (as can be seen by shifting $q$ by $ir! p$), so it
suffices to show that the above holds with $i=0$.
By \eqref{gamma-form}, we may rewrite this requirement as
$$ 
\#  \{ (p,q) \in (\PP_0 \backslash \PP_1(\vect{\mathbf{a}})) \times \QQ_0(\vect{\mathbf{a}}):  p \relra q \} = o\left( \gamma^r \alpha_r \frac{y}{\log^r x} \times \frac{x}{\log x} \right).
$$

Now from \eqref{po} and \eqref{pp1} we have
$$ 
\#  \{ (p,q) \in \PP_1(\vect{\mathbf{a}}) \times \QQ_0(\vect{\mathbf{a}}):  p \relra q \} \asym \gamma^r\alpha_r\frac y{\log^r x} \times \frac{x}{2\log x}  
$$
with probability at least $1-\eps$, so it suffices to show that
$$ 
\#  \{ (p,q) \in \PP_0 \times \QQ_0(\vect{\mathbf{a}}):  p \relra q \} \leq \frac{1+o(1)}{1-4\eps} \gamma^r\alpha_r\frac y{\log^r x} \times \frac{x}{2\log x}  
$$
with probability at least $4\eps$ (recall that $\eps=o(1)$).  By Markov's inequality, it thus suffices to show that
$$ 
\E \#  \{ (p,q) \in \PP_0 \times \QQ_0(\vect{\mathbf{a}}):  p \relra q \} \leq (1+o(1)) \gamma^r\alpha_r\frac y{\log^r x} \times \frac{x}{2\log x}.
$$
But this follows from the $b=1$ case of \eqref{est-1}.  The proof of Lemma \ref{sieveunited} is now complete.

\section{Linear equations in primes with large shifts}\label{dickson-shifted-sec}

The paper \cite{gt-linearprimes} of the second and fourth author is concerned with counting the number of prime points parameterized by a system of affine-linear forms in a convex body, when the constant terms in the affine-linear forms are comparable to the size of the body.
To establish Lemma \ref{first} we will require a strengthening of the main result in \cite{gt-linearprimes}, in which the constant terms in the affine-linear forms are permitted to be larger than the size of the body by a logarithmic factor. The aim of this section is to state this strengthening. The proof involves a number of minor modifications to the arguments of \cite{gt-linearprimes}: these are indicated in Appendix \ref{linear-primes-app}.

To state the results, we need to recall some notation from \cite{gt-linearprimes}.  If $d,t \geq 1$ be integers, then an \emph{affine-linear form} on $\Z^d$ is a function $\psi: \Z^d \to \Z$ which is the sum $\psi = \dot \psi + \psi(0)$ of a homogeneous linear form $\dot \psi: \Z^d \to \Z$ and a constant $\psi(0) \in \Z$.  A \emph{system of affine-linear forms} on $\Z^d$ is a collection $\Psi = (\psi_1,\ldots,\psi_t)$ of affine-linear forms on $\Z^d$. 
A system $\Psi$ is said to have finite complexity if and only if no form $\dot\psi_i$ is a multiple of any other form $\dot\psi_j$.

We recall that the \emph{von Mangoldt function} $\Lambda(n)$ is defined to equal $\log p$ when $n$ is a prime $p$ or a power of that prime, and zero otherwise.

Here is the main result of \cite{gt-linearprimes}.

\begin{gt-maintheorem}
  Let $N, d, t, L$ be positive integers, and let $\Psi = (\psi_1,\ldots,\psi_t)$ be a system of affine-linear forms of finite complexity with 
\begin{equation}\label{size-assumption}\Vert \Psi \Vert_N \leq L.  \end{equation}
Let $K \subset [-N,N]^d$ be a convex body.  Then we have
\begin{equation}\label{dickson-eq}
\sum_{ \vect{n} \in K \cap \Z^d} \prod_{i=1}^t \Lambda( \psi_i( \vect{n}) ) = \beta_\infty \prod_p \beta_p + o_{t,d,L}(N^d)
\end{equation}
where 
\[ \beta_{\infty} := \vol_d\big( K \cap \Psi^{-1}((\R^+)^t) \big)\]
and
\[ \beta_p := \E_{\vect{n} \in (\Z/p\Z)^d} \prod_{i=1}^t \Lambda_{\Z/p\Z}(\psi_i(\vect{n})).\]
Here $\Vert \Psi \Vert_N$ is defined by
\[ \Vert \Psi \Vert_N := \sum_{i=1}^t \sum_{j = 1}^d |\dot{\psi_i}(e_j)|  + \sum_{i = 1}^t \left|\frac{\psi_i(0)}{N}\right|.\]
The function $\Lambda_{\Z/p\Z}: \Z \to \R^+$ is the \emph{local von Mangoldt function}, that is the $p$-periodic function defined by setting $\Lambda_{\Z/p\Z}(b) := \frac{p}{p-1}$ when $b$ is coprime to $p$ and $\Lambda_{\Z/p\Z}(b) = 0$ otherwise.  Also, $\{e_1,\ldots,e_d\}$ is the standard basis for $\R^d$.
\end{gt-maintheorem}

Strictly speaking, the results in \cite{gt-linearprimes} were conditional on two (at the time unproven) conjectures, namely the M\"obius-Nilsequences conjecture and the inverse conjecture for the Gowers uniformity norms.  However, these conjectures have since been proven in \cite{gt-nilmobius} and \cite{GTZ} respectively, and so the above theorem is now unconditional.

The variant of this result that we shall need is that in which the condition \eqref{size-assumption} is replaced by the weaker condition 

\begin{equation}\label{non-hom-weak} \Vert \Psi \Vert_{N,B} \leq L, \end{equation}
where $B > 0$ is some constant (in fact any $B > 1$ will suffice for us). Here we have defined
\[ \Vert \Psi \Vert_{N,B} := \sum_{i=1}^t \sum_{j = 1}^d |\dot{\psi_i}(e_j)|  + \sum_{i = 1}^t \left|\frac{\psi_i(0)}{N\log^B N}\right|.\] Note that $\Vert \Psi \Vert_{N,0} = \Vert \Psi \Vert_N$.

 The conclusion is the same, except that the error term in \eqref{dickson-eq} must also depend on $B$.

\begin{thm}\label{dickson-shifted}
Let $B > 0$ be a positive quantity. Let everything be as in Theorem A, 
except assume that instead of condition \eqref{size-assumption} we have only the weaker condition \eqref{non-hom-weak}. Then we have
\[
\sum_{ \vect{n} \in K \cap \Z^d} \prod_{i=1}^t \Lambda( \psi_i(\vect{n}) ) = \beta_\infty \prod_p \beta_p + o_{t,d,L,B}(N^d),
\]
where $\beta_{\infty}$ and the $\beta_p$ are given by the same formulae as before.
\end{thm}

This extension in effect allows us to consider affine linear forms in
which the constant terms $\psi_i(0)$ can have size up to $\asymp N
\log^B N$, whereas in Theorem A, they are restricted to have size $O(N)$.  As mentioned above, the proof of Theorem \ref{dickson-shifted} is deferred to Appendix \ref{linear-primes-app}.

\section{Proof of Lemma \ref{first}(i)}\label{first-lemma-sec}

In this section we deduce Lemma \ref{first}(i) from Theorem \ref{dickson-shifted}. Throughout this section, $x$ and $y$ obey \eqref{xy}, all $o(1)$ terms may depend on $r$, and $\alpha_r$ is defined in \eqref{alpha-r-def}.

It suffices to prove the lemma when $x$ is an integer,
which we henceforth assume.
We first partition the range $(x/4,y]$ of $q$ into blocks of size about $x$, so that $p$ and
$q$ range over intervals of roughly the same size.
Namely, for a non-negative integer $m$ and $u \in \R$ we write
\[ I(m,u) := \Z \cap [mx, (m+1)x) \cap (x/4 , \infty) \cap [0, y - r!(r-1)u].\] Observe that
\begin{equation}\label{obs-1}
\sum_{0 \leq m \leq y/x} \# I(m,n_1) \asym y
\end{equation}
uniformly for $x/2 < n_1 \le x$ and that
\begin{equation}\label{obs-2} \# (m,n_1)= x \quad \hbox{ for all } x/2<n_1\le x
\end{equation} for all except $o(y/x)$ values of $m$, $0 \leq m \leq y/x$. We call these exceptional values of $m$ \emph{bad} and the remaining $0 \leq m \leq y/x$ obeying \eqref{obs-2} \emph{good}. Trivially $|I(m,n_1)| \leq x$ for all $m,n_1$.

We claim the following estimate:

\begin{prop}\label{sigma-prop}  We have
\begin{equation}\label{sig-def} 
\sum_{\substack{0 \leq m \leq y/x \\x/2 < n_1 \le x}} |F(m,n_1)|^2 \Lambda(n_1) = o(y x^2 )
\end{equation}
where
\begin{equation}\label{F-def} F(m,n_1) := \sum_{n_2 \in I(m,n_1)} \left(\prod_{j = 0}^{r-1} \Lambda(n_2 + jr! n_1) - \alpha_r\right).\end{equation}
\end{prop}

Let us assume this proposition for the moment and conclude the proof of Lemma \ref{first}(i).
Let $\eps=\eps(x)>0$ with $\eps$ decaying to zero sufficiently slowly.
If $n_1$ is a prime in $(x/2,x]$, say that $n_1$ is \emph{exceptional}
and write $n_1 \in \mathscr{E}$ if the number of $q$ for which $x/4 < q  < y - (r-1) r! n_1$ and $q + jr! n_1$ is prime for $j = 0,\dots, r-1$ differs from $\alpha_r y/\log^r x$ by at least $\eps y/\log^r x$. It follows straightforwardly
that if $n_1 \in \mathscr{E}$ then
\[ \Bigg| \sum_{x/4 \leq n_2 < y - (r-1) r! n_1} \prod_{j =
  0}^{r-1}\Lambda(n_2 + jr! n_1)  - \alpha_r y\Bigg| \geq
\frac{1}{2}\eps y\] if $x$ is sufficiently large. (To see
this, note that due to the restriction on the ranges of $n_1,n_2$,
$\Lambda(n_2 + j r! n_1) = \log x+O(\log_2 x)$ whenever $n_2 + jr!
n_1$ is prime. $\Lambda$ is also supported on prime powers, but the
contribution from these is negligible.)
Recall the definition \eqref{F-def} of $F(m,n_1)$.  Using
the fact that $[x/4, y - (r-1)r! n_1] = \bigcup_m I(m,n_1)$ and
\eqref{obs-1},
we conclude that
\[
\Bigg| \sum_{0\le m\le y/x} F(m,n_1) \Bigg| \ge \frac14 \eps y
\]
for sufficiently large $x$.  By Cauchy's inequality, we thus have
\[
 \sum_{0\le m\le y/x} |F(m,n_1)|^2 \ge \frac{\(\frac14 \eps
   y\)^2}{\frac{y}{x}+2}\ge \frac1{32} \eps^2 xy \qquad (n_1 \in  \mathscr{E}).
\]
Since $\Lambda(n_1) = \log n_1 \ge \log (x/2)$ for every prime $n_1$, we therefore see that the left-hand side of \eqref{sig-def} is at least
\[
\frac1{32} \eps^2 xy \log(x/2) \# \mathscr{E}.
\]
Applying \eqref{sig-def}, we conclude that $\# \mathscr{E}=o(x/\log x)$ if
$\eps$ goes to zero slowly enough, and Lemma \ref{first}(i) follows.

We now prove the proposition.  After a change of variables, the left-hand side of \eqref{sig-def} may be written as
\[
\sum_{\substack{0 \leq m \leq y/x \\x/2 < n_1 \le x}} \Bigg|\sum_{n_2 \in I(m, n_1) - mx } \Bigg(\prod_{j = 0}^{r-1}\Lambda(n_2 + jr! n_1 + mx)  - \alpha_r\Bigg) \Bigg|^2 \Lambda(n_1).\]
Expanding out the square, we can write this expression as
\[ \sum_{0 \leq m \leq y/x} \Sigma_2(m) - 2\alpha_r \sum_{0 \leq m
  \leq y/x} \Sigma_1(m) + \alpha_r^2 \sum_{0 \leq m \leq y/x}
\Sigma_0(m)\]
 where $\Sigma_2(m), \Sigma_1(m), \Sigma_0(m)$ are the quantities
\begin{align*}
\Sigma_2(m)  &:=  \sum_{\substack{x/2 < n_1 \le x \\n_2 \in I(m, n_1) - mx \\ n_3 \in I(m,n_1) - mx }}  \Lambda(n_1) \prod_{\substack{0 \leq j \leq r-1 \\ \ell = 2,3}}\Lambda(n_{\ell} + jr! n_1 + mx)\\
\Sigma_1(m) &:=  \sum_{\substack{x/2 < n_1 \le x \\ n_2 \in I(m, n_1) - mx }} (\# I(m,n_1)) \Lambda(n_1)  \prod_{j = 0}^{r-1} \Lambda(n_2 + jr! n_1 + mx)\\
\Sigma_0(m) &:= \sum_{x/2 < n_1 \le x} (\# I(m,n_1))^2 \Lambda(n_1).
\end{align*}
To prove \eqref{sig-def}, it will thus suffice to establish the estimates
\begin{equation}\label{sigo}
\sum_{0 \leq m \leq y/x} \Sigma_b(m) \asym \alpha_r^b \frac{yx^2}{2}
\end{equation}
for $b=0,1,2$.

We begin with the $b=2$ case, which is the most difficult.  We apply Theorem \ref{dickson-shifted} with $d := 3$, $t := 2r + 1$, and the forms $\Psi = (\psi_1,\dots, \psi_{2r+1})$ given by
\[ \Psi(n_1,n_2,n_3) := (n_1,  (n_{\ell} + jr! n_1 + mx)_{0 \leq j \leq r-1, \ell =
  2,3})\] and convex polytope $K = K(m)$ given by
\[ K(m) := \{ (u_1,u_2,u_3) \in \R^3 : x/2 < u_1 \le x, u_2,u_3 \in I(m,u_1) - mx \}. \]
 Since $\Psi(K(m)) \subset (\R^{+})^{2r + 1}$, it follows from Theorem \ref{dickson-shifted} that
\begin{equation}\label{sig-2m} \Sigma_2(m) = \vol(K(m)) \prod_p \beta_p + o(x^3),\end{equation}
where
\[ \beta_p := \E_{\vect{n} \in (\Z/p\Z)^3} \prod_{i = 1}^{2r +1} \Lambda_{\Z/p\Z}(\psi_i(\vect{n})).\]
Obviously the system $\Psi$ has finite complexity.

We claim that
\begin{equation}\label{bp-claim} \beta_p = \left\{ \begin{array}{ll} (\frac{p}{p-1})^{2(r - 1)} & p \leq r \\ \pfrac{(p-r)p^{r-1}}{(p-1)^r}^2 & p > r.\end{array}  \right. \end{equation}
The proof of the claim is quite straightforward. Indeed if $p \leq r$ then, modulo $p$, $n_2 + jr! n_1 + mx \equiv n_2 + mx$ and $n_3 + jr! n_1 + mx \equiv n_3 + mx$, and so all the forms $\psi_i(\vect{n})$ are coprime to $p$ if and only if none of $n_1,n_
2 + mx$ or $n_3 + mx$ is zero mod $p$. Thus the number of $\vect{n} = (n_1, n_2,n_3)$ for which all of the forms $\psi_i(\vect{n})$ are nonzero mod $p$ is precisely $(p-1)^3$.

If, by contrast, $p > r$ then either $n_1 \equiv 0 \pmod{p}$ or else the values of $n_2+ jr! n_1 + mx$, $0 \leq j < r$ are all distinct mod $p$, and hence at most one of them can be zero. The same is true for the values of $n_3 + jr! n_1 + mx$. Thus if $n_
1 \not\equiv 0 \pmod{p}$ then there are $r$ values of $n_2$ for which one of the forms $\psi_i(\vect{n})$ vanishes, and also $r$ values of $n_3$ for which one of these forms vanishes, and thus $2rp - r^2$ pairs $(n_2,n_3)$ in total. Thus in this case the number of $\vect{n} = (n_1, n_2,n_3)$ for which all of the forms $\psi_i(\vect{n})$ are nonzero mod $p$ is $p^3 - p^2 - (p-1)(2rp - r^2) = (p-1)(p-r)^2$, and this confirms the formula for $\beta_p$.

It follows from the claim \eqref{bp-claim} and the definition \eqref{alpha-r-def} of $\alpha_r$ that $\prod_p \beta_p = \alpha_r^2$ and hence, by \eqref{sig-2m}, that
\[ \Sigma_2(m) = \vol(K(m)) \alpha_r^2 + o(x^3).\]
By \eqref{obs-2} above we have $\vol(K(m)) = x^3/2$ for all good values of $m$, and $\vol(K(m)) \leq x^3/2$ for all $m$. It is thus straightforward to conclude the required asymptotic \eqref{sigo} for $b=2$.

Next we turn to the $b=1$ case of \eqref{sigo}. Define
\[ S_1(m) := \sum_{\substack{x/2 < n_1 \le x \\ 0 \leq n_2 < x}}\Lambda(n_1) \prod_{j = 0}^{r-1} \Lambda(n_2 + jr! n_1 + mx).\] Then, by \eqref{obs-2},
\begin{equation}\label{sigma2-s2} x \sum_{m\, \mbox{\scriptsize good}} S_1(m) \leq \sum_m \Sigma_1(m) \leq x\sum_{0 \leq m \leq y/x} S_1(m).\end{equation}
 To estimate $S_1(m)$, apply Theorem \ref{dickson-shifted} with $d := 2$, $t := r+1$, forms $\Psi = (\psi_1,\dots, \psi_{r+1})$ given by
\[ \Psi(n_1,n_2) := (n_1,(n_2 + jr! n_1 + mx)_{0 \leq j < r})\] and convex polytope $K := (x/2, x] \times [0,x)$. The system $\Psi$ also has finite complexity. Noting that $\Psi( K) \subset (\R^{+})^{r+1}$, we obtain

\begin{equation}\label{sig-2-single} S_1(m) = \frac{x^2}{2} \prod_p \beta_p + o(x^2)\end{equation} uniformly in $m$
where
\[ \beta_p := \E_{\vect{n} \in (\Z/p\Z)^2} \prod_{i = 1}^{r+1} \Lambda_{\Z/p\Z} (\psi_i(\vect{n})).\]
We claim that
\begin{equation}\label{bp2} \beta_p = \left\{ \begin{array}{ll} (\frac{p}{p-1})^{r-1} & p \leq r \\ \frac{(p-r)p^{r-1}}{(p-1)^r} & p > r.\end{array}  \right. \end{equation}

The proof of the claim is similar to that of \eqref{bp-claim} but rather easier. Indeed if $p \leq r$ then, modulo $p$, $n_2 + jr! n_1 + mx \equiv n_2 + mx$, and so all the forms $\psi_i(\vect{n})$ are coprime to $p$ if and only if neither $n_1$ nor $n_2 +
 mx$ is zero mod $p$, and so the number of $\vect{n} = (n_1, n_2)$ for which all of the forms $\psi_i(\vect{n})$ are nonzero mod $p$ is precisely $(p-1)^2$.

If $p > r$ then either $n_1 \equiv 0 \pmod{p}$ or else the values of $n_2 + jr! n_1 + mx$, $0 \leq j < r$ are all distinct mod $p$, and hence at most one of them can be zero. Thus if $n_1 \not\equiv 0 \pmod{p}$ then there are $r$ values of $n_2$ for which one
 of the forms $\psi_i(\vect{n})$ vanishes. Thus in this case the number of $\vect{n} = (n_1, n_2)$ for which all of the forms $\psi_i(\vect{n})$ are nonzero mod $p$ is $p^2 - p - (p-1)r = (p-1)(p-r)$, and this confirms the formula for $\beta_p$.

From \eqref{bp2} and \eqref{alpha-r-def} we have $\prod_p \beta_p =
\alpha_r$. It follows from \eqref{sigma2-s2}, \eqref{sig-2-single}
and \eqref{bp2} that \eqref{sigo} holds for $b=1$.

Finally we establish the $b=0$ case of \eqref{sigo}. By \eqref{obs-2}, for all except $o(y/x)$ bad values of $m$ we have $\# I(m,n_1) = x$. If $m$ is good then by the prime number theorem $\Sigma_0(m) \asym x^3/2$, and so the contribution to
$\sum_m \Sigma_0(m)$ from the good $m$ is $\asym yx^2/2$. The contribution from the bad $m$ can be absorbed into the error term, and so \eqref{sigo} for $b=0$ follows.  The proof of Lemma \ref{first}(i) is now complete.

\section{Proof of Lemma \ref{first}(ii)}\label{second-lemma-sec}

In this section we deduce Lemma \ref{first}(ii) from Theorem \ref{dickson-shifted}. The argument is very similar to that in the last section.  As before, $x$ and $y$ obey \eqref{xy}, all $o(1)$ terms may depend on $r$, and $\alpha_r$ is defined in \eqref{alpha-r-def}.

We may again assume that $x$ is an integer.  The analogue of Proposition \ref{sigma-prop} is

\begin{prop}\label{sigma-prop-2}  We have
\begin{equation}\label{sig-new-def} 
\sum_{x/4 < n_1 \le y} |F(n_1)|^2 \Lambda(n_1) = o(y x^2 )
\end{equation}
where
\begin{equation}\label{F-new-def} F(n_1) := \sum_{x/2<n_2\le x} \left(\Lambda(n_2) \prod_{\substack{-i \leq j < r - i \\ j \neq 0}}\Lambda(n_1 + j r! n_2)  - \alpha_r\right).\end{equation}
\end{prop}

Let us assume this proposition for the moment and conclude the proof of Lemma \ref{first}(ii).
Let $\eps = \eps(x) > 0$ tend to 0 as $x\to\infty$ sufficiently
slowly. 
If $n_1$ is a prime in $(x/4,y]$, we say that $n_1$ is
  \emph{exceptional} and write $n_1 \in \mathscr{E}$ if the number of
  primes $p \in (x/2,x]$ for which $n_1 + j r! p$ is a prime in
    $(x/4,y]$ for $-i \leq j < r - i$, $j \neq 0$, differs from
      $\alpha_r (x/2)/\log^r x$ by at least $\eps x/\log^r x$. 
Arguing as in the proof of Lemma \ref{first}(i), if $n_1
      \in \mathscr{E}$ then for sufficently large $x$ we have
\begin{equation}\label{w33} \Bigg| \sum_{\substack{x/2<n_2\le x \\ n_1 - ir! n_2 > x/4 \\ n_1 + (r - i - 1)r! n_2 \leq y}} \prod_{\substack{-i \leq j < r-i \\ j \neq 0}} \Lambda(n_1 + jr! n_2) - \frac{1}{2}\alpha_r x\bigg| \geq \frac{1}{2}\eps x.\end{equation}
Note that the second and third conditions in the summation are
precisely what constrain all the $n_1 + jr! n_2$, $-i \leq j < r-i$,
to lie in $(x/4,y]$.  If we assume that
\[
(r+1)! x < n_1 < y - (r+1)! x
\]
and recall from \eqref{F-new-def} above the definition of $F(n_1)$,
we see that \eqref{w33} is equivalent to 
\[ |F(n_1)| \geq \frac{1}{2}\eps x.\]
Since $\Lambda(n_1) = \log n_1 \ge \log (x/4)$ for all prime $n_1$, we conclude from
the prime number theorem that the left-hand side of \eqref{sig-new-def} is at least
\[
 \left( \frac12 \eps x \right)^2 \log(x/4) \big( \# \mathscr{E} - O(x/\log x) \big).
\]
From this and \eqref{sig-new-def} we conclude that  $\# \mathscr{E}=o(y/\log
x)$  provided $\eps$ tends to zero sufficiently slowly, and Lemma \ref{first}(ii) follows.

It remains to establish Proposition \ref{sigma-prop-2}.  For this, we break up the range of
$n_1$ as in the proof of Lemma \ref{first}~(i).
 For a non-negative integer $m$ define
\[ I(m) := \Z \cap [mx,(m+1)x) \cap (x/4, y].\]
Then we may decompose the left-hand side of \eqref{sig-new-def} as
\[ \sum_{\substack{0 \leq m \leq y/x\\ n_1 \in I(m)}}  |F(n_1)|^2\Lambda(n_1),\]
 With a simple change of variables we see that this quantity equals
\[ \sum_{\substack{0 \leq m \leq  y/x \\ n_1 \in I(m) - mx}} \left| \sum_{x/2<n_2\le x} \!\! \Lambda(n_2)\!\!\!\! \prod_{\substack{-i \leq j < r-i \\ j \neq 0}} \!\!\Lambda(n_1 +j r! n_2 + mx)  - \frac{1}{2}\alpha_r x\right |^2\Lambda(n_1 + mx). \]
Expanding out the square and applying the prime number theorem, we may therefore express the above quantity as
\[ \sum_{0\leq m \leq y/x} \Sigma_2(m) - \alpha_r x \sum_{0 \leq m
  \leq y/x} \Sigma_1(m) + \frac{1}{4}\alpha_r^2 x^2 y + o(x^2 y),\] 
where
\[ \Sigma_2(m) :=  \sum_{\substack{n_1 \in I(m) - mx \\ x/2<n_2 \le x
    \\ x/2 < n_3 \le x}}  \Lambda(n_1 + mx) \Lambda(n_2)\Lambda(n_3)
\!\!\!\prod_{\substack{-i \leq j < r - i \\ j \neq 0 }} \;
\prod_{\ell=2,3} \Lambda(n_1 + j r! n_{\ell} + mx) \]
and
\[ \Sigma_1(m) = \sum_{\substack{n_1 \in I(m) - mx \\ x/2<n_2\le x}}  \Lambda(n_2)\prod_{-i \leq j < r - i}\Lambda(n_1 +j r! n_2 + mx) .\]
It will thus suffice to show that
\begin{equation}\label{sigo-new}  \sum_{0 \leq m \leq y/x} \Sigma_b(m) \asym \left(\alpha_r \frac{x}{2}\right)^b y
\end{equation}
for $b=1,2$.

We first handle the $b=2$ case of \eqref{sigo-new}.  
We can estimate $\Sigma_2(m)$ using Theorem \ref{dickson-shifted} with $d := 3$, $t := 2r + 1$, forms $\Psi = (\psi_1,\dots, \psi_{2r+1})$ given by
\[ \Psi(n_1,n_2,n_3) := (n_1 + mx, n_2,n_3, (n_1 + jr! n_{\ell} + mx)_{-i\leq j < r-i, j \neq 0, \ell= 2,3} )\] and convex polytope $K(m) := (I(m) - mx) \times (x/2, x] \times (x/2,x]$. The theorem tells us that uniformly in $m$ we have
\begin{equation}\label{sig-1-prelim-new} \Sigma_2(m) = \vol(K(m)) \prod_p \beta_p + o(x^3),\end{equation} where again
\[ \beta_p := \E_{\vect{n} \in (\Z/p\Z)^3} \prod_{i = 1}^{2r +1} \Lambda_{\Z/p\Z}(\psi_i(\vect{n})).\]
It is again clear that the system $\Psi$ has finite complexity. 

Now we claim that the $\beta_p$ are given by the same formulae as in \eqref{bp-claim}, that is to say
\[ \beta_p = \left\{ \begin{array}{ll} (\frac{p}{p-1})^{2(r - 1)} & p
  \leq r \\ \pfrac{(p-r)p^{r-1}}{(p-1)^r}^2 & p > r.\end{array}
\right. \]
The proof of this is very similar to that of \eqref{bp-claim}, but subtly different. If $p \leq r$ then the forms $\psi_i(\vect{n})$ are all equal to one of $n_1+ mx,n_2,n_3$ mod $p$, and so there are $(p-1)^3$ choices of $\vect{n} \in (\Z/p\Z)^3$ for which all the forms are coprime to $p$. If $p > r$ then we must choose $n_1 \not\equiv -mx \pmod{p}$. For any such choice there are precisely $r$ choices of $n_2$ for which one of $n_1 + jr! n_2 + mx$ ($-i \leq j < r - i$, $j \neq 0$) and $n_2$ is $0 \pmod{p}$, namely $n_2 \equiv 0\pmod{p}$ and $n_2 \equiv -(jr!)^{-1} (n_1 + mx) \pmod{p}$ for $-i \leq j < r-i$, $j \neq 0$. Similarly there are precisely $r$ choices for which one of $n_1 + jr! n_3 + mx$ ($-i \leq j < r-i$, $j \neq 0$) and $n_3$ is $0 \pmod{p}$, and so we have $2rp - r^2$ bad choices of $(n_2,n_3)$ for each $n_1 \not\equiv -mx \pmod{p}$. Therefore, as before, the number of choices of $\vect{n}$ for which at least one of the $\psi_i(\vect{n})$ vanishes mod $p$ is $p^3 - p^2 - (2rp - p^2) = (p-1)(p-r)^2$.

Therefore $\prod_p \beta_p = \alpha_r^2$, and hence from \eqref{sig-1-prelim-new} we have
\[ \sum_{0 \leq m \leq y/x} \Sigma_2(m) = \alpha_r^2 \sum_{0 \leq m \leq y/x} \vol(K(m)) + o(y x^2).\] 
We have $\# I(m) = x$ and hence $\vol(K(m)) = x^3/4$ for all except $o(y/x)$ values of $m$, and so the $b=2$ case of \eqref{sigo-new} follows immediately.

Now we turn our attention to the $b=1$ case of \eqref{sigo-new}. Again we can estimate it using Theorem \ref{dickson-shifted}, now with $d := 2$, $t := r+1$, forms $\Psi = (\psi_1,\dots, \psi_{r+1})$ given by 
\[ \Psi(n_1,n_2) := ( n_2, (n_1 + jr! n_2 + mx)_{-i \leq j < r-i})\] and convex polytope $K(m) := (I(m) - mx) \times (x/2,x]$. 

 The theorem tells us that uniformly in $m$ we have
\begin{equation}\label{sig-2-prelim-new} \Sigma_1(m) = \vol(K(m)) \prod_p \beta_p + o(x^2),\end{equation} where 
\[ \beta_p := \E_{\vect{n} \in (\Z/p\Z)^2} \prod_{i = 1}^{r+1} \Lambda_{\Z/p\Z}(\psi_i(\vect{n})).\]
Again the system $\Psi$ has finite complexity. 

We claim that the $\beta_p$ are the same as in \eqref{bp2}, that is to say
\[ \beta_p = \left\{ \begin{array}{ll} (\frac{p}{p-1})^{r-1} & p \leq
  r \\ \frac{(p-r)p^{r-1}}{(p-1)^r} & p > r.\end{array}  \right. \]
Indeed if $p \leq r$ then, mod $p$, all the forms in $\Psi$ are either $n_1 + mx$ or $n_2$, so there are $(p-1)^2$ choices of $\vect{n}$ for which all of the $\psi_i(\vect{n})$ are coprime to $p$. If $p > r$ then we must take $n_1 \not\equiv -mx \pmod{p}$, and then for each such choice there are precisely $r$ values of $n_2 \pmod{p}$ for which one of the $\psi_i(\vect{n})$ is $0 \pmod{p}$, namely $n_2 \equiv 0\pmod{p}$ and $n_2 \equiv -(jr!)^{-1} (n_1 + mx) \pmod{p}$ for $-i \leq j < r-i$, $j \neq 0$. It follows that there are $p^2 - p - r(p-1) = (p-1)(p-r)$ choices of $\vect{n} \in (\Z/p\Z)^2$ for which none of the $\psi_i(\vect{n})$ is $0 \pmod{p}$.

Therefore $\prod_p \beta_p = \alpha_r$, and hence from \eqref{sig-2-prelim-new} we have
\[ \sum_{0 \leq m \leq y/x} \Sigma_1(m) = \alpha_r \sum_{0 \leq m \leq y/x} \vol(K(m)) + o(y x).\]
We have $\# I(m) = x$ and hence $\vol(K(m)) = x^2/2$ for all except $o(y/x)$ values of $m$, and so the $b=1$ case of \eqref{sigo-new} follows immediately.  The proof of Lemma \ref{first}(ii) is now complete.

\section{Further comments and speculations}

The reduction of Theorem \ref{third-red} to Theorem \ref{fourth-red} was somewhat wasteful, as one replaced the entire residue class $q_p \pmod{p}$ by a fairly short arithmetic progression $q_p, q_p+r!p,\dots, q_p+(r-1)r!p$ inside that residue class.  One could attempt to strengthen the argument here by working with more general patterns such as $q_p, q_p + a_1 r! p, \dots, q_p + a_r r! p$ for some $0 < a_1 < \dots < a_r = o(y/x)$, and possibly trying to exploit further averaging over the $a_1,\dots,a_r$.  However, we were unable to take advantage of such ideas to make any noticeable improvements to the arguments or results.

The dependence of $R$ on $x$ in Theorem \ref{mainthm} is completely ineffective, for two different reasons. The sources of this ineffectivity are
\begin{itemize}
\item The use of Davenport's ineffective bound
\[ \sup_{\theta} |\E_{n \in [N]} \mu(n) e(n \theta)| \ll_A \log^{-A} N\] in \cite{gt-nilmobius}, which is intimately related to the possibility of Siegel zeros; and
\item the use of ultrafilter arguments in \cite{GTZ} (and in other work of the inverse conjectures for the Gowers norms, such as that of Szegedy \cite{szegedy}).
\end{itemize}

The first source of ineffectivity appears to be less serious than the second with our present state of knowledge. For example, if one is only interested in having the conclusion of Theorem \ref{mainthm} for an infinite sequence of $x$'s (rather than all sufficiently large $x$) then by choosing $x$ judiciously the influence of Siegel zeros can be avoided and one has an effective version of Davenport's bound.  See \cite{Da} for some related discussion.

The second source of ineffectivity is problematic, since in taking $R$ large we need inverse theorems for the Gowers $U^{s+1}[N]$-norm with $s = s(R)$ tending to infinity. Proofs not using ultrafilters are only known in the cases $s = 2,3$ and $4$, and the bounds in the inverse theorem \cite{GTZ-4} for the Gowers $U^{4}[N]$-norm (which were not worked out in that paper) are already incredibly bad, of ``$\log_*$ type'' or worse. In principle (but with great pain) the ultrafilters in \cite{GTZ} could be removed, but the bounds would be similarly bad. It seems that a genuinely new idea is needed to make these bounds, and thus the approach of the present paper, effective in any reasonable sense (for example $R$ being bounded below by $\log_k x$ for some finite $k$).

\appendix

\section{Linear equations in primes}\label{linear-primes-app}

\emph{In this appendix all page numbers refer to the published version of the paper \cite{gt-linearprimes}, with which we assume a certain familiarity.}

We turn now to the proof of Theorem \ref{dickson-shifted}, indicating the points at which we must be careful assuming only the bound $\Vert \Psi \Vert_{N, B} \leq L$ rather than the stronger bound $\Vert \Psi \Vert_{N} \leq L$ allowed in Theorem A, which is the main theorem of \cite{gt-linearprimes}.  The key points are that (a) the sieve-theoretic portions of \cite{gt-linearprimes} are essentially unaffected by shifts, and (b) the M\"obius-nilsequences conjecture used in \cite{gt-linearprimes} comes with a savings of $\log^{-A} N$ for arbitrary $A > 0$, which is enough to absorb the effect of shifting for that portion of the argument.

We require a precise measure of the \emph{complexity} of the system
 $\Psi$ (cf. \cite[Definition 1.5]{gt-linearprimes}) which plays a crucial role in the arguments.  If $1 \leq i \leq t$ and $s \geq 0$ then we say that $\Psi$ has $i$-complexity at most $s$ if one can cover the $t-1$ forms $\{\dot\psi_j : j \ne i \}$ by $s + 1$ classes, such that $\dot\psi_i$ does not lie in the linear span of any of these classes. The \emph{complexity} of the system of forms $\Psi$ is defined to be the least $s$ for which the system has $i$-complexity at most $s$ for all $1 \leq i \leq t$, or $\infty$ if no such $s$ exists.  Note that a system $\Psi$ has finite complexity if and only if no form $\dot\psi_i$ is a multiple of any other form $\dot\psi_j$. 

 Let us first of all note that \cite{gt-linearprimes} was written to be conditional upon two sets of conjectures, the M\"obius and Nilsequences Conjecture $\mbox{MN}(s)$ and the Inverse Conjectures for the Gowers norms $\mbox{GI}(s)$ which were unproven at the time in the cases $s \geq 3$. These are now theorems, established in \cite{gt-nilmobius} and \cite{GTZ} respectively, and thus the results of \cite{gt-linearprimes} which we plan to modify in this section are unconditional. We have no need to change any aspect of the inner workings of either \cite{gt-nilmobius} or \cite{GTZ}.

The argument in \cite{gt-linearprimes} proceeds via a series of reductions to other statements. First, in \cite[Chapter 4]{gt-linearprimes}, some straightforward linear algebra reductions are given. The first part of the chapter concerns \cite[Theorem 1.8]{gt-linearprimes} and does not concern us here; our interest begins near the top of page 1771. The subsection ``\emph{Elimination of the archimedean factor}'' makes no use of any bound on $\Vert \Psi \Vert_N$. This section allows us to assume henceforth that $\psi_1,\dots,\psi_t > N^{9/10}$ on $K$. The only change we need to make to the next subsection, ``\emph{Normal form reduction of the main theorem}'' is to replace $\Vert \cdot \Vert_N$ in the statement of \cite[Lemma 4.4]{gt-linearprimes} by $\Vert \cdot \Vert_{N, B}$. That such a variant is valid follows from the proof of \cite[Lemma 4.4]{gt-linearprimes} and in particular the observation that $\tilde\Psi(0) = \Psi(0)$, where $\tilde\Psi : \Z^{d'} \rightarrow \Z^t$ is the system of forms constructed in that proof.

The rest of \cite[Chapter 4]{gt-linearprimes} carries over unchanged. Thus (changing $L$ to $\tilde L = O_{d,t,L}(1)$) we may assume henceforth that our system affine-linear forms $\psi_i$ is in $s$-normal form and still satisfies $\Vert \Psi \Vert_{N,B} \leq L$.

The next step, undertaken in \cite[Chapter 5]{gt-linearprimes} is to decompose the sum
\[ \sum_{\vect{n} \in K \cap \Z^d} \prod_{i \in [t]} \Lambda(\psi_i(\vect{n}))\] along progressions with common difference $W = \prod_{p \leq w} p$, where $w = \log \log \log N$ (say). This is the ``$W$-trick''. The task of proving Theorem A is reduced to that of establishing the estimate (\cite[Theorem 5.1]{gt-linearprimes})
\begin{equation*}\label{gt-51} \sum_{\vect{n} \in K \cap \Z^d} \Bigg(\prod_{i \in [t]} \Lambda'_{b_i, W}(\psi_i(\vect{n})) - 1\Bigg) = o(N^d)\end{equation*}
with $b_1,\dots,b_t \in [W]$ coprime to $W$, uniformly in the choice of $b_i$. Here
\[ \Lambda'_{b,W}(n) := \frac{\phi(W)}{W}\Lambda'(Wn + b)\] and $\Lambda'$ denotes the restriction of $\Lambda$ to the primes.

We claim that the proof of Theorem \ref{dickson-shifted} may be similarly reduced to the task of establishing 

\begin{equation}\label{gt-51-new} \sum_{\vect{n} \in K \cap \Z^d} \Bigg(\prod_{i \in [t]} \Lambda'_{b_i, W}(\psi_i(\vect{n})) - 1\Bigg) = o_B(N^d)\end{equation} uniformly for $b_1,\dots, b_t \in [W]$, but now only assuming the weaker condition $\Vert \Psi \Vert_{N, B} \leq L$.

The reduction proceeds exactly as in \cite[Chapter 5]{gt-linearprimes}, except that at the bottom of page 1777 we must remark that the constant term $\tilde\psi_{i,a}(0)$ is now only bounded by $O_{L,d,t}(N\log^B N/W)$, and where on page 1778 we said that $\Vert \tilde\Psi \Vert_{\tilde N} = O(1)$, we must now say that $\Vert \tilde\Psi \Vert_{\tilde N, B} = O(1)$.

The desired estimate \eqref{gt-51-new} may be written in the equivalent form
\begin{equation}\label{gt-51-new2} \sum_{\vect{n} \in K \cap \Z^d} \Bigg(\prod_{i \in [t]} T^{\psi_i(0)}\Lambda'_{b_i, W}(\dot\psi_i(\vect{n})) - 1\Bigg) = o_B(N^d),\end{equation} where $\dot\psi_i$ denotes the homogeneous (linear) part of the affine form $\psi_i$ and $T$ denotes the translation operator defined by $T^af(x) := f(x + a)$. The homogeneous system $\dot\Psi = (\dot\psi_1,\dots, \dot\psi_t)$ satisfies the condition $\Vert \tilde\Psi \Vert_N \leq L$.

The first step in proving \eqref{gt-51-new2} is to prove a variant of \cite[Proposition 6.4]{gt-linearprimes} for the shifted functions $T^{\psi_i(0)} \Lambda'_{b_i, W}$. We claim that in fact the following generalisation of that proposition holds (for notation and further discussion, see \cite[Chapter 6]{gt-linearprimes}).

\begin{prop64new}\label{shift-majorant}
Let $D > 1$ be arbitrary, and let $z_1,\dots, z_t \in \Z_{\geq 0}$, $z_i \leq N^{1.01}$, be arbitrary shifts. Then there is a constant $C_0 := C_0(D)$ such that the following is true. Let $C \geq C_0$, and suppose that $N' \in [CN,2CN]$. Let $b_1,\dots, b_t \in \{0,1,\dots, W-1\}$ be coprime to $W$. Then there exists a $D$-pseudorandom measure $\nu : \Z_{N'} \rightarrow \R^{+}$ \textup{(}depending on $z_1,\dots, z_t$\textup{)} which obeys the pointwise bounds
\[ 1 + T^{z_1}\Lambda'_{b_1, W}(n) + \dots + T^{z_t}\Lambda'_{b_t, W}(n) \ll_{D,C} \nu(n)\] for all $n \in [N^{3/5}, N]$, where we identify $n$ with an element of $\Z_{N'}$ in the obvious manner.
\end{prop64new}

The proof of \cite[Proposition 6.4]{gt-linearprimes} was presented in \cite[Appendix D]{gt-linearprimes}. We now indicate the modifications necessary to that argument to obtain the more general Proposition 6.4'. The first modification we need to make is on page 1839, where we instead define the preliminary weight $\tilde\nu : [N] \rightarrow \R^+$ by setting
\[ \tilde\nu(n) := \E_{i \in [t]} \frac{\phi(W)}{W} T^{z_i}\Lambda_{\chi,R,2}(Wn + b_i).\] 
We have the bound
\begin{equation}\label{analogous-bound} T^{z_i}\Lambda'_{b_i,W}(n)
  \ll_{C,D} \frac{\phi(W)}{W} T^{z_i}\Lambda_{\chi, R, 2}(Wn +
  b_i)\end{equation} for all $i \in [t]$ and all $n \in [N^{3/5}, N]$,
analogous to that stated at the bottom of page 1839. The key
observation here is that the left-hand side is only nonzero when $W(n
+ z_i) + b_i$ is a prime, in which case it equals
$\frac{\phi(W)}{W}\log(W (n + z_i) + b_i) < \frac{2\phi(W)}{W} \log N$
(since $W \le \log N, n \leq N$ and $z_i \leq N^{1.01}$). However if $n \in [N^{3/5}, N]$ then $W(n + z_i) + b_i \geq N^{3/5}$, and so if the sieve level $\gamma$ used in the definition of $\Lambda_{\chi, R,2}$ satisfies $\gamma < \frac{3}{5}$ then the right-hand side is $\frac{\phi(W)}{W} \log R$. Since $R = N^{\gamma}$ and $\gamma$ depends only on $C,D$ (see halfway up page 1839), \eqref{analogous-bound} follows.

As in \cite[Appendix D]{gt-linearprimes}, we then transfer to $\Z_{N'}$ by setting $\nu(n) := \frac{1}{2} + \frac{1}{2} \tilde\nu(n)$ when $n \in [N]$ and $\nu(n) := 1$ otherwise. 

We then need to go back and modify the proof of \cite[Theorem D.3]{gt-linearprimes} so that it applies with $T^{z_i}\Lambda_{\chi_i, R, a_i}$ replacing $\Lambda_{\chi_i, R, a_i}$. Equivalently, we need to establish this theorem with only the weak bound $|\psi_i(0)| \ll N^{1.01}$ on the constant terms of the forms $\psi_i$, rather than the stronger bound $\Vert \Psi \Vert_N \leq L$ assumed there.  In fact, no bound on the $\psi_i(0)$ is required in this part of the argument at all. The first place in that argument that the assumption $\Vert \Psi \Vert_N \leq L$ is used is in page 1833, where it is asserted that $\alpha(p,B) = \E_{\vect{n} \in \Z_p^d} 1_{p | \psi_i(\vect{n})} $ is equal to $1/p$ if $p \geq p_0(t,d,L)$ is sufficiently large. It is easy to see that the bound here depends only on the sizes of the coefficients in the homogeneous parts of $\psi_i$. The second place that this assumption is used is on page 1834, in the appeal to \cite[Lemma 1.3]{gt-linearprimes}. As it happens only two of the three conclusions of this lemma as stated are valid under the weaker assumption: there is a superfluous statement about what happens for $p > C(d,t,L) N$ which fails in our present context, but which is not needed for the applications in \cite[Appendix D]{gt-linearprimes}. An appropriately modified version of the lemma is the following.

\begin{lem13new} Suppose that $\Psi = (\psi_1,\dots, \psi_t)$ is a system of linear forms such that the homogeneous parts $\dot\Psi = (\dot\psi_1,\dots,\dot\psi_t)$ satisfy $\Vert \dot\Psi\Vert_N \leq L$. Then the local factors $\beta_p$ satisfy $\beta_p = 1 + O_{t,d,L}(p^{-1})$. If, furthermore, no two of the forms $\dot\psi_i$ are parallel then $\beta_p = 1 + O_{t,d,L}(p^{-2})$.
\end{lem13new}

This lemma, whose proof is the same as that of \cite[Lemma 1.3]{gt-linearprimes}, applies equally well on page 1834. The rest of the proof of \cite[Theorem D.3]{gt-linearprimes} goes through unchanged.

The deduction of the linear forms conditions for $\tilde\nu$ now proceeds exactly as on pages 1840, with $\Lambda_{\chi_i,R,a_i}$ replaced by its shifted variant $T^{z_i} \Lambda_{\chi_i, R, a_i}$ whenever necessary.

The proof of the correlation conditions for $\tilde \nu$, starting at the bottom of page 1840, needs to be tweaked a little\footnote{Note, however, that by the work of Conlon, Fox and Zhao \cite{cfz} one could in principle dispense with the need for this condition entirely.}. Instead of the bound at the bottom of page 1840, we must establish a variant with shifts, namely
\[ \left(\frac{\phi(W)}{W}\right)^m \Bigg(\sum_{n \in I}\prod_{j \in [m]} \Lambda_{\chi, R, 2}(W(n + h_j)+ b_{i_j} + Wz_{i_j})  \Bigg) \ll N \sum_{1 \leq j < j' \leq m}\tau(h_j - h_{j'})\] whenever $i_1,\dots, i_m \in [t]$. Here, the function $\tau$ is required to satisfy $\E_{n \in [-N,N]} \tau(n)^q \ll_q 1$. In the argument on page 1841, the set $P_{\Psi}$ is now the set of primes dividing $W(h_j - h_{j} + z_{i_j} - z_{i_j'}) + b_{i_j} - b_{i_{j'}}$ for some $1 \leq j < j' \leq m$, and we define
\[ \tau(n) := \sum_{1 \leq j < j' \leq m} \exp \Bigg(O  (1)\sum_{\substack{p > w \\ p | Wn + W(z_{i_j} - z_{i_{j'}}) + (b_{i_j} - b_{i_{j'}})}} \frac{1}{p^{1/2} }\Bigg) .\]
It now suffices to prove the bound
\[ \E_{n \in [N]} \exp\Bigg(q \sum_{\substack{p > w \\ p | Wn + h}} \frac{1}{p^{1/2}}\Bigg) \ll_q 1\] uniformly for all $h = O(N^{1.02})$. This is the same as the estimate at the bottom of page 1841, only there we had the stronger assumption $h = O(W)$. The only difference this makes to the argument is that the third displayed equation on page 1842 (which it is our task to prove) only comes with the weaker constraint $d = O(N^{1.02})$, that is to say we must show
\[ \sum_{\substack{(d,W) = 1 \\ d = O(N^{1.02})}} d^{-1/4} \sum_{\substack{n \in [N] \\ d | Wn + h}} 1 \ll N,\] whereas before we had $d = O(WN)$. However, the proof of this slightly stronger bound is the same: using the bound
\[ \sum_{\substack{n \in [N] \\ d | Wn + h}} 1 \ll 1 + \frac{N}{d},\] it reduces to
\[ \sum_{d = O(N^{1.02})} d^{-1/4} \left(1 + \frac{N}{d}\right) \ll N,\] a true statement. This at last completes the proof of Proposition 6.4'.

We now continue with the arguments of \cite[Chapter 7]{gt-linearprimes}. Using Proposition 6.4' in place of \cite[Proposition 6.4]{gt-linearprimes}, we see that the proof of \eqref{gt-51-new}, and hence of Theorem \ref{dickson-shifted}, reduces to establishing the bound

\begin{equation*}\label{gowers-shift} \Vert T^z \Lambda'_{b,W} - 1\Vert_{U^{s+1}[N]} = o_{s,B}(1)\end{equation*} uniformly for all $b \in \{0,1,\dots, W-1\}$ and for all shifts $z$ with $|z| \leq LN\log^B N$. 

By the arguments of \cite[Section 10]{gt-linearprimes} (but using Proposition 6.4' in place of \cite[Proposition 6.4]{gt-linearprimes}) we can reduce to proving the bound

\begin{equation*}\label{nil-cor} \E_{n \in [N]} (T^z \Lambda'_{b,W}(n) - 1) \psi(n) = o_{\psi,B}(1)\end{equation*} for any $s$-step nilsequence $\psi(n) = F(g^n x)$, where the $o_{\psi}(1)$ term may depend on the underlying nilmanifold $G/\Gamma$ and the Lipschitz constant of $F$ but not on the nilrotation $g$.

Chapter 11 of \cite{gt-linearprimes} requires no change, and the only changes required to Chapter 12 up to the bottom of page 1804 are to replace $\Lambda^{\sharp}$ and $\Lambda^{\flat}$ by their shifted variants $T^z\Lambda^{\sharp}$ and $T^z\Lambda^{\flat}$. This reduces matters to establishing the two estimates
\begin{equation}\label{lam-sharp} \left\Vert \frac{\phi(W)}{W} T^z\Lambda^{\sharp}(Wn + b) - 1\right\Vert_{U^{s+1}[N]} = o_s(1)  \end{equation} (the shifted analogue of (12.5) in \cite{gt-linearprimes}) 
and
\begin{equation}\label{lam-flat}  \E_{n \in [N]} \frac{\phi(W)}{W} T^z \Lambda^{\flat}(Wn + b) \psi(n) = o_{\psi, B}(1)\end{equation} for all nilsequences $\psi$ (the shifted analogue of (12.4) in \cite{gt-linearprimes}.

The proof of the first of these, \eqref{lam-sharp}, proceeds exactly as in the proof of (12.5) of \cite{gt-linearprimes}, which is given on page 1842--1843. The only change required is to use the variant of \cite[Theorem D.3]{gt-linearprimes} with shifts, the validity of which was noted above. For this argument, we do not need any bound on $z$.

Finally we turn to the estimate \eqref{lam-flat}. The analysis of page 1805 may be easily adapted, with the result that it is enough to prove that 
\[ \E_{n \in [N]} T^{zW} \Lambda^{\flat}(n) \psi(n) = o_{\psi, B}(1).\]
This, however, follows immediately from (12.10) of \cite{gt-linearprimes}, which asserted the bound
\[ \bigg|\sum_{n \in [N]} \Lambda^{\flat}(n) \psi(n)\bigg| \ll_{\psi, A} N \log^{-A} N\] for any $A$. In particular, taking $A = B + 2$ (and noting that $W = o(\log N)$ and $z \leq LN\log^B N$) we have
\[ \bigg|\sum_{1 \leq n \leq N + zW} \Lambda^{\flat}(n) \psi(n)\bigg| = o_{\psi,B}(N)\]
and
\[ \bigg|\sum_{1 \leq n \leq N} \Lambda^{\flat}(n)\psi(n)\bigg| = o_{\psi, B}(N).\]
Subtracting these two estimates gives the result.

\end{document}